\documentclass[a4paper]{amsart}
\usepackage{geometry}
\usepackage[foot]{amsaddr}
\usepackage{amsthm,amssymb}

\usepackage{tikz}
\usetikzlibrary{shapes,arrows,positioning}

\usepackage[utf8]{inputenc}
\usepackage{amsmath,amssymb,amsthm,commath,esint,tikz-cd,tikz, mathtools, physics}
\usepackage[mathscr]{euscript}
\usepackage{amsfonts}
\usepackage{subcaption}
\usepackage{hyperref}
\usepackage{graphicx}
\usepackage{float}
\usepackage{fancyhdr}
\usepackage{bm}
\usepackage{bbm}

\newcommand{\infn}[1]{{\left\vert #1 
    \right\vert_\infty}}
\numberwithin{equation}{section}

\title[Error Analysis for PINNs approximating Kolmogorov PDEs]{Error Analysis for Physics Informed Neural Networks\\ (PINNs) approximating Kolmogorov PDEs}

\author{T.~De Ryck}
\email{tim.deryck@sam.math.ethz.ch}

\author{S.~Mishra}
\address[T. De Ryck and S. Mishra]{Seminar for Applied Mathematics, ETH Z\"urich, R\"amistrasse 101, 8092 Z\"urich, Switzerland}
\email{siddhartha.mishra@sam.math.ethz.ch}
\newcommand{\s}{s} %dimension of flux
\renewcommand{\S}{\mathcal{S}} %training set
\newcommand{\Et}{\mathcal{E}_T} %training error
\newcommand{\Eg}{\mathcal{E}_G} %generalization error
\newcommand{\Aet}{\overline{\mathcal{E}}_T} %acc training error
\newcommand{\Aeg}{\overline{\mathcal{E}}_G} %acc gen error

\newcommand{\hu}{\hat{u}}
\newcommand{\bigO}{\mathcal{O}}
\newcommand{\E}[1]{{\mathbb{E}\left[ #1 \right]}} %expectation
\newcommand{\Prob}[1]{{\mathbb{P}\left( #1 \right)}} %probability

\newcommand{\vpa}{\widehat{\varphi}_\xi}

\newcommand{\mult}{\widehat{\times}_\eta}
\newcommand{\unn}{\widehat{u}^{M,N}}
\newcommand{\utilde}{\Tilde{u}^{M,N}}
\newcommand{\ubar}{\bar{u}^{N}}
\newcommand{\F}[1]{(\mathcal{F}\varphi)\left( #1 \right)}
\newcommand{\Fhat}[1]{(\widehat{\mathcal{F}\varphi})_\delta\left( #1 \right)}
\renewcommand{\L}[1]{\mathcal{L}\left[ #1 \right]}
\newcommand{\nO}{\frac{n_0(t)T}{N}}
\newcommand{\EL}{\mathcal{D}}

\newtheorem{theorem}{Theorem}[section]

\newtheorem{remark}[theorem]{Remark}
\newtheorem{definition}[theorem]{Definition}
\newtheorem{lemma}[theorem]{Lemma}

\newtheorem{corollary}[theorem]{Corollary}

\begin{document}

\begin{abstract}
Physics informed neural networks approximate solutions of PDEs by minimizing pointwise residuals. We derive rigorous bounds on the error, incurred by PINNs in approximating the solutions of a large class of linear parabolic PDEs, namely Kolmogorov equations that include the heat equation and Black-Scholes equation of option pricing, as examples. We construct neural networks, whose PINN residual (generalization error) can be made as small as desired. We also prove that the total $L^2$-error can be bounded by the generalization error, which in turn is bounded in terms of the training error, provided that a sufficient number of randomly chosen training (collocation) points is used. Moreover, we prove that the size of the PINNs and the number of training samples only grow polynomially with the underlying dimension, enabling PINNs to overcome the curse of dimensionality in this context. These results enable us to provide a comprehensive error analysis for PINNs in approximating Kolmogorov PDEs.  
\end{abstract}

\maketitle

%\tableofcontents
\section{Introduction}
\noindent {\bf Background and context.}
Partial differential equations (PDEs) are ubiquitous as mathematical models in the sciences and engineering. Explicit solution formulas for PDEs are not available except in very rare cases. Hence, numerical methods, such as finite difference, finite element and finite volume methods, are key tools in approximating solutions of PDEs. In spite of their well-documented successes, it is clear that these methods are inadequate for a variety of problems involving PDEs. In particular, these methods are not suitable for efficiently approximating PDEs with \emph{high-dimensional} state or parameter spaces. Such problems arise in different contexts ranging from PDEs such as the Boltzmann, Radiative transfer, Schr\"odinger and Black-Scholes type equations with very high number of spatial dimensions, to \emph{many-query} problems, as in uncertainty quantification (UQ), optimal design and inverse problems, which are modelled by PDEs with very high parametric dimensions. 

Given this pressing need for efficient algorithms to approximate the afore-mentioned problems, machine learning methods are being increasingly deployed in the context of scientific computing. In particular, deep neural networks (DNNs) i.e., multiple compositions of affine functions and scalar nonlinearities, are being widely used. Given the \emph{universality} of DNNs in being able to approximate any continuous (measurable) function to desired accuracy, they can serve as ansatz spaces for solutions of PDEs, as for high-dimensional semi-linear parabolic PDEs \cite{HEJ1}, linear elliptic PDEs \cite{SZ1,Kuty} and nonlinear hyperbolic PDEs \cite{LMR1,LMPR1} and references therein. More recently, DNN-inspired architectures such as DeepOnets \cite{ChenChen,DeepOnet,LMK1} and Fourier Neural operators \cite{FNO} have been shown to even learn infinite-dimensional \emph{operators}, associated with underlying PDEs, efficiently.

A large part of the literature on the use of deep learning for approximating PDEs relies on the \emph{supervised learning} paradigm, where the DNN has to be \emph{trained} on possibly large amounts of labelled data. However in practice, such data is acquired from either measurements or computer simulations. Such simulations might be very computationally expensive \cite{LMR1} or even infeasible in many contexts, impeding the efficiency of the supervised learning algorithms. Hence, it would be very desirable to find a class of machine learning algorithms that can approximate PDEs, either without any explicit need for data or with very small amounts of data. Physics informed neural networks (PINNs) provide exactly such a framework. \\

\noindent {\bf Physics Informed Neural Networks (PINNs).} PINNs were first proposed in the 90s \cite{DPT,Lag1,Lag2} as a machine learning framework for approximating solutions of differential equations. However, they were resurrected recently in \cite{KAR1,KAR2} as a practical and computationally efficient paradigm for solving both forward and inverse problems for PDEs. Since then, there has been an explosive growth in designing and applying PINNs for a variety of applications involving PDEs. A very incomplete list of references includes \cite{KAR4,KAR5,KAR6,KAR7,KAR8,jag1,jag2,MM1,MM2,MM3,BKMM1} and references therein. 

We briefly illustrate the idea behind PINNs by considering the following general form of a PDE:
\begin{equation}\label{eq:PDE}
    \EL[u](x,t) = 0, \quad \mathcal{B}u(y,t) = \psi(y,t), \quad u(x,0) = \varphi(x), \quad \text{for } x\in D, y\in \partial D, t\in [0,T], 
\end{equation}
Here, $D\subset \mathbb{R}^d$ is compact and $\EL,\mathcal{B}$ are the differential and boundary operators, $u:D\times [0,T]\to\mathbb{R}^m$ is the solution of the PDE, $\psi:\partial D\times [0,T]\to\mathbb{R}^m$ specifies the (spatial) boundary condition and $\varphi:D\to\mathbb{R}^m$ is the initial condition.

We seek deep neural networks $u_{\theta}:D\times [0,T]\to\mathbb{R}^m$ (see \eqref{eq:dnn} for a definition), parameterized by $\theta \in \Theta$, constituting the weights and biases, that approximate the solution $u$ of \eqref{eq:PDE}. To this end, the key idea behind PINNs is to consider pointwise \emph{residuals}, defined for any sufficiently smooth function $f:D\times [0,T]\to\mathbb{R}^m$ as, 
\begin{equation}\label{eq:pinn-residuals}
    \mathcal{R}_i[f](x,t) = \EL[f](x,t), \quad   \mathcal{R}_s[f](t,y) = \mathcal{B}f(t,y) - \psi(t,y), \quad \mathcal{R}_t[f](x) = f(0,x) - \varphi(x)
\end{equation}
for $x\in D$, $y\in \partial D$, $t\in [0,T]$. Using these residuals, one measures how well a function $f$ satisfies resp. the PDE, the boundary condition and the initial condition of \eqref{eq:PDE}. Note that for the exact solution $ \mathcal{R}_i[u]=\mathcal{R}_s[u]=\mathcal{R}_t[u]=0$. 

Hence, within the PINNs algorithm, one seeks to find a neural network $u_\theta$, for which all residuals are simultaneously minimized, e.g. by minimizing the quantity,
\begin{equation}\label{eq:def-generalization-error}
    \Eg(\theta)^2 = \int_{D\times[0,T]} \abs{\mathcal{R}_i[u_\theta](x,t)}^2 dxdt +  \int_{\partial D\times[0,T]} \abs{\mathcal{R}_s[u_\theta](x,t)}^2 ds(x)dt +  \int_{D} \abs{\mathcal{R}_t[u_\theta](x)}^2 dx. 
\end{equation}
However, the quantity $\Eg(\theta)$, often referred to as the \emph{population risk} or  \textit{generalization error} \cite{MM1} of the neural network $u_\theta$ involves integrals and can therefore not be directly minimized in practice. Instead, the integrals in \eqref{eq:def-generalization-error} are approximated by a numerical quadrature, resulting in,
\begin{equation}
     \Et^i(\theta,\S_i)^2 = \sum_{n=1}^{N_i} w^n_i \abs{\mathcal{R}_i[u_\theta](t^n_i,x^n_i)}^2, \quad
      \Et^s(\theta,\S_s)^2 = \sum_{n=1}^{N_s} w^n_s \abs{\mathcal{R}_s[u_\theta](t^n_s,x^n_s)}^2, \quad
      \Et^t(\theta,\S_t)^2 = \sum_{n=1}^{N_t} w^n_t \abs{\mathcal{R}_t[u_\theta](x^t_i)}^2.
\end{equation}
Here, one samples quadrature points in space-time to construct data sets $\S_i = \{(t^n_i,x^n_i)\}_{n}^{N_i}$, $\S_s = \{(t^n_s,x^n_s)\}_{n}^{N_s}$ and $\S_t = \{x^n_t\}_{n}^{N_t}$, and $w^n_q$ are suitable quadrature weights for $q=i,t,s$.
Thus, the \emph{generalization error} $\Eg(\theta)$ is approximated by the so-called \emph{training loss} or \emph{training error} \cite{MM1},
\begin{equation}\label{eq:def-training-error}
    \Et(\theta,\S)^2 = \Et^i(\theta,\S_i)^2 + \Et^s(\theta,\S_s)^2+ \Et^t(\theta,\S_t)^2, 
\end{equation}
where $\S = (\S_i, \S_s, \S_t)$, and a stochastic gradient descent algorithm is to used to approximate the non-convex optimization problem,
\begin{equation}
\label{eq:opt}
    \theta^* = \arg\min_{\theta\in\Theta}  \Et(\theta,\S)^2, 
\end{equation}
and $u^{\ast} = u_{\theta^{\ast}}$ is the trained PINN that approximates the solution $u$ of the PDE \eqref{eq:PDE}. \\

\noindent {\bf Theory for PINNs.}
Given this succinct description of the PINNs algorithm, the following fundamental theoretical questions arise immediately,
\begin{itemize}
    \item [Q1.] Given a tolerance $\varepsilon > 0$, does there exist a PINN $\hat{u} = u_{\hat{\theta}}$, parametrized by a $\hat{\theta} \in \Theta$ such that the corresponding generalization error (population risk) $\Eg(\hat{\theta})$\label{Q1} \eqref{eq:def-generalization-error} is small i.e., $\Eg(\hat{\theta}) < \varepsilon$?
    \item [Q2.] Given a PINN $\hat{u}$ with small generalization error, is the corresponding \emph{total error} $\|u-\hat{u}\|$ small, i.e., is $\|u-\hat{u}\|< \delta (\varepsilon)$, for some $\delta(\varepsilon) \sim {\mathcal O}(\varepsilon)$, for some suitable norm $\|.\|$, and with $u$ being the solution of the PDE \eqref{eq:PDE}?\label{Q2}   
\end{itemize}
The above questions are of fundamental importance as affirmative answers to them certify that, \emph{in principle}, there exists a PINN, corresponding to the parameter $\hat{\theta}$, such that the resulting PDE residual \eqref{eq:pinn-residuals} is small, and consequently also the overall error in approximating the solution of the PDE \eqref{eq:PDE}.

Moreover, the smallness of the generalization error $\Eg(\hat{\theta})$ can imply that the training error $\Et(\hat{\theta})$ \eqref{eq:def-training-error}, which is an approximation of the generalization error, is also small. Hence, \emph{in principle}, the (global) minimization of the optimization problem \eqref{eq:opt} should result in a proportionately small training error. 

However, the optimization problem \eqref{eq:opt} involves the minimization of a \emph{non-convex}, very-high dimensional objective function. Hence, it is unclear if a global minimum is attained by a gradient-descent algorithm. \emph{In practice}, one can evaluate the training error $\Et(\theta^{\ast})$ for the (local) minimizer $\theta^{\ast}$ of \eqref{eq:opt}. Thus, it is natural to ask if,
\begin{itemize}
    \item [Q3.] Given a small training error $\Et(\theta^{\ast})$ and a sufficiently large training set $\S$, is the corresponding generalization error $\Eg(\theta^{\ast})$ also proportionately small? \label{Q3}
\end{itemize}
An affirmative answer to question Q3, together with question Q2, will imply that the trained PINN $u_{\theta^*}$ is an accurate approximation of the solution $u$ of the underlying PDE \eqref{eq:PDE}. 
Thus, answering the above three questions affirmatively will constitute a comprehensive theoretical investigation of PINNs and provide a rationale for their very successful empirical performance. 

Given the very large number of papers exploring PINNs empirically, the rigorous theoretical study of PINNs is in a relative state of infancy. In \cite{shin2020convergence}, the authors prove a consistency result for PINNs, for linear elliptic and parabolic PDEs, where they show that if $\Et(\theta_m)\to 0$ for a sequence of neural networks $\{u_{\theta_m}\}_{m\in\mathbb{N}}$, then $\norm{u_{\theta_m}-u}_{L^\infty}\to0$, under the assumption that one adds a specific $C^{k,\alpha}$-regularization term to the loss function, thus partially addressing question Q3 for these PDEs. However, this result does not provide quantitative estimates on the underlying errors. A similar result, with more quantitative estimates for advection equations is provided in \cite{Zhang1}.  

In \cite{MM1,MM2}, the authors provide a strategy for answering questions Q2 and Q3 above. They leverage the \emph{stability} of solutions of the underlying PDE \eqref{eq:PDE} to bound the total error in terms of the generalization error (question Q2). Similarly, they use accuracy of quadrature rules to bound the generalization error in terms of the training error (question Q3). This approach is implemented for Forward problem corresponding to a variety of PDEs such as the semi-linear and quasi-linear parabolic equations and the incompressible Euler (Navier-Stokes) equations \cite{MM1}, radiative transfer equations \cite{MM3}, nonlinear dispersive PDEs such as the KdV equations \cite{BKMM1} and for the unique continuation (data assimilation) inverse problem for many linear elliptic, parabolic and hyperbolic PDEs \cite{MM2}. However, these works suffer from two essential limitations: first, question Q1 on the smallness of generalization error is not addressed and second, the assumptions on the quadrature rules in \cite{MM1,MM2} are rather stringent and in particular, the analysis does not include the common choice of using random sampling points in $\S$, unless an additional validation set is chosen. Thus, the theoretical analysis presented in \cite{MM1,MM2} is incomplete and this sets the stage for the current paper. \\

\noindent {\bf Aims and scope of this paper.} Given the above discussion, our main aims in this paper are to address the fundamental questions Q1, Q2 and Q3 and to establish a solid foundation and rigorous rationale for PINNs in approximating PDEs. 

To this end, we choose to focus on a specific class of PDEs, the so-called Kolmogorov equations \cite{oksendal2003stochastic} in this paper. These equations are a class of \emph{linear, parabolic} PDEs which describe the space-time evolution of the density for a large set of stochastic processes. Prototypical examples include the heat (diffusion) equation and Black-Scholes type PDEs that arise in option pricing. A key feature of Kolmogorov PDEs is the fact that the equations are set in very high dimensions. For instance, the spatial dimension in a Black-Scholes PDE is given by the number of underlying assets (stocks), upon which the basket option is contingent, and can range up to hundreds of dimensions.  

Our motivation for illustrating our analysis on Kolmogorov PDEs is two-fold. First, they offer a large class of PDEs with many applications, while still being linear. Second, it has already been shown empirically in \cite{MM1,RT,MMT1} that PINNs can approximate very high-dimensional Kolmogorov PDEs efficiently. 

Thus in this paper, 
\begin{itemize}
    \item We show that there exist PINNs, approximating a class of Kolmogorov PDEs, such that the resulting generalization error \eqref{eq:def-generalization-error}, and the total error, can be made as small as possible. Moreover under suitable hypothesis on the initial data and the underlying exact solutions, we will show that the size of these PINNs does not grow exponentially with respect to the spatial dimension of the underlying PDE. This is done by explicitly constructing PINNs using a representation formula, the so-called Dynkin's formula, that relates the solutions of the Kolmogorov PDE to the generator and sample paths for the underlying stochastic process. 
    \item We leverage the stability of Kolmogorov PDEs to bound the error, incurred by PINNs in $L^2$-norm in approximating solutions of Kolmogorov PDEs, by the underlying generalization error. 
    \item We provide rigorous bounds for the generalization error of the PINN approximating Kolmogorov PDEs in terms of the underlying training error \eqref{eq:def-training-error}, provided that the number of \emph{randomly} chosen training points is sufficiently large. Furthermore, the number of random training points does not grow exponentially with the dimension of the underlying PDE. We use a novel error decomposition and standard Hoeffding's inequality type covering number estimates to derive these bounds.
\end{itemize}
Thus, we provide affirmative answers to questions Q1, Q2 and Q3 for this large class of PDEs. Moreover, we also show that PINNs can \emph{overcome the curse of dimensionality} in approximating these PDEs. Hence, our results will place PINNs for these PDEs on solid theoretical foundations. 

The rest of the paper is organized as follows: In section \ref{sec:2}, we present preliminary material on linear Kolmogorov equations and describe the PINNs algorithm to approximate them. The generalization error and total error (questions Q1 and Q2) are considered in section \ref{sec:3} and the generalization error is bounded in terms of training error (question Q3) in section \ref{sec:gen}. 
\section{PINNs for Linear Kolmogorov Equations}
\label{sec:2}
\subsection{Linear Kolmogorov PDEs}\label{sec:Kolmogorov}
In this paper, we will consider the following general form of linear time-dependent partial differential equations,
\begin{equation}\label{eq:kolmogorov-pde}
    \begin{cases}u_t(t,x) = \frac{1}{2}\text{Trace}(\sigma(x)\sigma(x)^T H_x[u](t,x)) + \mu(x)^T\cdot \nabla_x[u](t,x) &\text{for all }(t,x)\in [0,T]\times D,\\ u(0,x) = \varphi(x) &\text{for all } x\in D, \\u(t,x) = \psi(x,t) &\text{for all } (t,x)\in [0,T]\times \partial D. \end{cases}
\end{equation}
where $\sigma:\mathbb{R}^d\to \mathbb{R}^{d\times d}$ and $\mu:\mathbb{R}^d\to \mathbb{R}^d$ are affine functions, $\nabla_x$ denotes the gradient and $H_x$ the Hessian (both with respect to the space coordinates). For definiteness, we set $D = (0,1)^d$. PDEs of the form \eqref{eq:kolmogorov-pde} are referred to as Kolmogorov equations and arise in a large number of models in science and engineering. 
Prototypical examples of Kolmogorov PDEs include,
\begin{itemize}
    \item [1.] {\bf Heat Equation}: Let $\mu=0$ and $\sigma = \sqrt{\kappa} I_d$, where $\kappa>0$ is the thermal diffusivity of the medium and $I_d$ is the $d$-dimensional identity matrix. This results in the following PDE for the temperature $u$,
    \begin{equation}
    \label{eq:ht}
        u_t(t,x) = \kappa \sum_{j=1}^d u_{x_jx_j}(t,x), \qquad u(0,x) = \varphi(x). 
    \end{equation}
    Here, $\varphi$ describes the initial heat distribution. Dirichlet or Neumann boundary data complete the problem. 
    \item [2.] {\bf Black-Scholes equation}: If both $\mu$ and $\sigma$ in \eqref{eq:kolmogorov-pde} are linear functions, we obtain the Black-Scholes equation, which models the evolution in time $t$ of the price of an option $u$ that is based on $d$ underlying stocks $x_i$. Up to a straightforward change of variables, the corresponding PDE is given by (see e.g. \cite{oksendal2003stochastic}),
    \begin{equation}
    \label{eq:BS}
         u_t(t,x) = \sum_{i,j=1}^d \beta_i\beta_j\rho_{ij}x_ix_ju_{x_ix_j}(t,x) + \sum_{j=1}^d \mu x_j u_{x_j}(t,x), \qquad u(0,x) = \varphi(x). 
    \end{equation}
    Here, the $\beta_i$ are stock volatilities, the coefficients $\rho_{ij}$ model the correlation between the different stock prices, $\mu$ is an interest rate and the initial condition $\varphi$ is interpreted as a payoff function. Prototypical examples of such payoff functions are $\varphi(x) = \max\{\sum_i a_ix_i-K,0\}$ (basket call option),  $\varphi(x) = \max\{\max_i a_ix_i-K,0\}$ (call on max) and analogously for put options. 
\end{itemize}
Our goal in this paper is to approximate the classical solution $u$ of Kolmogorov equations with PINNs. We start with a brief recapitulation of neural networks below. 
\subsection{Neural Networks.}
\label{sec:dnn}
We denote by  $\sigma:\mathbb{R}\to\mathbb{R}$ be an (at least) twice continuously differentiable activation function, like tanh or sigmoid. For any $n\in\mathbb{N}$, we write for $z\in\mathbb{R}^n$ that $\sigma(z) := (\sigma(z_1), \ldots, \sigma(z_n))$. We formally define a neural network below,
\begin{definition}
\label{def:nn-app}
Let $R\in(0,\infty]$, $L,W\in\mathbb{N}$ and $l_0,\ldots, l_L\in\mathbb{N}$. Let $\sigma:\mathbb{R}\to\mathbb{R}$ be a twice differentiable function and define 
\begin{equation}
   \Theta =  \Theta_{L,W,R} := \bigcup_{L'\in\mathbb{N}, L'\leq L}\:\bigcup_{l_0,\ldots,l_L\in\{1, \ldots, W\}}\bigtimes_{k=1}^{L'} \left([-R,R]^{l_k\times l_{k-1}}\times[-R,R]^{l_k}\right). 
\end{equation}
For $\theta\in\Theta_{L,W,R}$, we define $(W_k,b_k):=\theta_k$ and $\mathcal{A}_k^\theta:\mathbb{R}^{l_{k-1}}\to\mathbb{R}^{l_{k}}:z\mapsto W_k z+b_k$ for $1\leq k\leq L$ and we define $f^\theta_k:\mathbb{R}^{l_{k-1}}\to\mathbb{R}^{l_{k}}$ by
\begin{equation}
    f_k^\theta(z) = \begin{cases}\mathcal{A}_L^\theta(z) & k=L,\\ (\sigma \circ \mathcal{A}_{k}^\theta)(z) & 1\leq k < L. \end{cases}
\end{equation}
We denote by $u_\theta:\mathbb{R}^{l_0}\to\mathbb{R}^{l_L}$ the function that satisfies for all $z \in\mathbb{R}^{l_0}$ that
% \begin{equation}
%     u_\theta(x) = \begin{cases}\mathcal{A}_1(x) & L=1\\ (\mathcal{A}_L \circ \sigma \circ \mathcal{A}_{L-1} \circ \sigma \circ\cdots \circ \sigma \circ \mathcal{A}_1)(x) & L>1. \end{cases}
% \end{equation}
\begin{equation}
\label{eq:dnn}
     u_\theta(z) = \left(f_{L}^\theta\circ f_{L-1}^\theta \circ \cdots \circ f_1^\theta\right)(z), 
\end{equation}
where in the setting of approximating Kolmogorov PDEs \eqref{eq:kolmogorov-pde} we set $l_0 = d+1$ and $z=(x,t)$.

We refer to $u_\theta$ as the realization of the neural network associated to the parameter $\theta$ with $L$ layers with widths $(l_0,l_1, \ldots, l_L)$, of which the middle $L-1$ layers are called hidden layers. For $1\leq k\leq L$, we say that layer $k$ has width $l_k$ and we refer to $W_k$ and $b_k$ as the weights and biases corresponding to layer $k$. If $L\geq 3$, we say that $u_\theta$ is a deep neural network (DNN). 
\end{definition}
\subsection{PINNs.}\label{sec:PINNs}
As already mentioned in the introduction, the key idea behind PINNs is to minimize pointwise residuals associated with the Kolmogorov PDE \eqref{eq:kolmogorov-pde}. To this end, we define the differential operator associated with \eqref{eq:kolmogorov-pde},
\begin{equation}
    \label{eq:L}
     \L{v}(x) = \sum_{i=1}^d \mu_i(x) (\partial_i v) (x,t) + \frac{1}{2} \sum_{i,j,k=1}^d \sigma_{ik}(x)\sigma_{kj}(x) (\partial^2_{ij}v)(x),
\end{equation}
for any $v \in C^2({\mathbb R}^d)$. 
Next, we define the following residuals associated with \eqref{eq:kolmogorov-pde},
\begin{equation}
    \label{eq:res}
    \begin{aligned}
    \mathcal{R}_i[v](x,t) &= \partial_t v(x,t) - \L{v}(x,t), \quad (x,t) \in D \times [0,T], \\
    \mathcal{R}_s[v](y,t) &= v(y,t) - \psi(y,t),\quad (y,t) \in \partial D \times [0,T],  \\
    \mathcal{R}_t[v](x) &= v(0,x) - \varphi(x), \quad \forall x \in D.
    \end{aligned}
\end{equation}
The \emph{generalization error} for a neural network of the form \eqref{eq:dnn}, approximating the Kolmogorov PDE is then given by the formula \eqref{eq:def-generalization-error}, but with the residuals defined in \eqref{eq:res}. 

Given the possibly very high-dimensional domain $D $ of \eqref{eq:kolmogorov-pde}, it is natural to use random sampling points to define the loss function for PINNs $\theta \mapsto \Et(\theta, \S)^2$ as follows, 
\begin{equation}
\label{eq:lf2}
\begin{aligned}
     \Et^i(\theta,\S_i)^2 &= \frac{1}{N_i}\sum_{n=1}^{N_i} \abs{\mathcal{R}_i[u_\theta](t^n_i,x^n_i)}^2, \\
      \Et^s(\theta,\S_s)^2 &= \frac{1}{N_s}\sum_{n=1}^{N_s} \abs{\mathcal{R}_s[u_\theta](t^n_s,x^n_s)}^2, \quad
      \Et^t(\theta,\S_t)^2 = \frac{1}{N_t} \sum_{n=1}^{N_t} \abs{\mathcal{R}_t[u_\theta](x^t_i)}^2,\\
      \Et(\theta,\S)^2 &= \Et^i(\theta,\S_i)^2 +  \Et^s(\theta,\S_s)^2+ \Et^t(\theta,\S_t)^2,
      \end{aligned}
\end{equation}
where the training data sets, $\S_i = \{(t^n_i,x^n_i)\}_{n}^{N_i}$, $\S_s = \{(t^n_s,x^n_s)\}_{n}^{N_s}$ and $\S_t = \{x^n_t\}_{n}^{N_t}$, are chosen randomly, independently with respect to the corresponding Lebesgue measures and the residuals $\mathcal{R}_{i,s,t}$ are defined in \eqref{eq:res}. 

A \emph{trained PINN} $u^{\ast} = u_{\theta^{\ast}}$ is then defined as a (local) minimum of the optimization problem \eqref{eq:opt}, with loss function \eqref{eq:lf2} (possibly with additional data and weight regularization terms), found by a (stochastic) gradient descent algorithm such as ADAM or L-BFGS. 
\section{Bounds on the approximation error for PINNs}
\label{sec:3}
In this section, we will first answer the question Q1 for the PINNs approximating linear Kolmogorov equations \eqref{eq:kolmogorov-pde} i.e., our aim will be to construct a deep neural network \eqref{eq:dnn} for approximating \eqref{eq:kolmogorov-pde}, such that the corresponding generalization error $\Eg$ \eqref{eq:def-generalization-error} is as small as desired. 

Recalling that the Kolmogorov PDE is a linear parabolic equation with smooth coefficients, one can use standard parabolic theory to conclude that there exists a unique classical solution $u$ of \eqref{eq:kolmogorov-pde} and it is sufficiently regular, for instance $u \in W^{s,\infty}((0,T)\times D)$ for some $s > 2$. 
As $u$ is a classical solution, the residuals \eqref{eq:res}, evaluated at $u$, vanish i.e.,
\begin{equation}
    \label{eq:resu}
\mathcal{R}_i[u](x,t) = 0, \quad \mathcal{R}_s[u](y,t) =0, \quad \mathcal{R}_t[u](x,0) = 0,
\end{equation}
for all $x \in D, y \in \partial D$.

Moreover, one can use recent results in approximation theory, such as those presented in \cite{guhring2020error,guhring2021approximation, deryck2021approximation} and references therein, to infer that one can find a deep neural network \eqref{eq:dnn} that approximates the solution $u$ in the $W^{2,\infty}$-norm, and therefore yields an approximation for which the PINN residual is small. For instance, one appeals to the following theorem (more details, including exact constants and bounds on the network weights, can be derived from the results in \cite{deryck2021approximation}).
\begin{theorem}
\label{thm:1}
Let $T>0$, $\gamma, d,s\in\mathbb{N}$ with $s\geq 2+\gamma$ and let $u\in W^{s,\infty}([0,T] \times [0,1]^d)$ be the solution of a linear Kolmogorov PDE \eqref{eq:kolmogorov-pde}. Then for every $\varepsilon>0$ there exists a tanh neural network $\widehat{u}^\varepsilon = u_{\widehat{\theta}^\varepsilon}$ with two hidden layers of width at most $\bigO(\varepsilon^{-d/(s-2-\gamma)})$ such that $\Eg(\widehat{\theta}^\varepsilon) \leq\varepsilon.$
\end{theorem}
\begin{proof}
It follows from \cite[Theorem 5.1]{deryck2021approximation} that there exists a tanh neural network $\widehat{u}^\varepsilon $ with two hidden layers of width at most $\bigO(\varepsilon^{-d/(s-2-\gamma)})$ such that 
\begin{equation}
    \norm{u-\widehat{u}^\varepsilon}_{W^{2,\infty}([0,T] \times [0,1]^d)} \leq \varepsilon.
\end{equation}
By virtue of the nature of linear Kolmogorov PDEs \eqref{eq:kolmogorov-pde} it follows immediately that $\norm{\mathcal{R}_i[u]}_{L^2([0,T] \times [0,1]^d)}\leq \varepsilon$. Using a standard trace inequality, one finds similar bounds for the $\mathcal{R}_s[u]$ and $\mathcal{R}_t[u]$. From this, it follows directly that $\Eg(\widehat{\theta}^\varepsilon) \leq\varepsilon.$ 
\end{proof}

Hence, $\widehat{u}^{\varepsilon}$ is a neural network for which the generalization error \eqref{eq:def-generalization-error} can be made arbitrarily small, providing an affirmative answer to Q1. 
However from Theorem \ref{thm:1}, we observe that the size (width) of the resulting deep neural network $\widehat{u}^{\varepsilon}$, grows \emph{exponentially} with spatial dimension $d$ for \eqref{eq:kolmogorov-pde}. Thus, this neural network construction clearly suffers from the \emph{curse of dimensionality}. Hence, this construction cannot explain the robust empirical performance of PINNs in approximating Kolmogorov equations \eqref{eq:kolmogorov-pde} in very high spatial dimensions \cite{MM1,RT,MMT1}. 
Therefore, we need a different approach for obtaining bounds on the generalization error that overcome this curse of dimensionality. To this end, we rely on the specific structure of the Kolmogorov equations \eqref{eq:kolmogorov-pde}. In particular, we will use the Dynkin's formula, which relates Kolmogorov PDEs to Itô diffusion SDEs.

In order to state Dynkin's formula, we first need to introduce some notation. Let $(\Omega, \mathcal{F}, P, (\mathbb{F}_t)_{t\in[0,T]})$ be a stochastic basis, $D\subseteq \mathbb{R}^d$ a compact set and, for every $x\in D$, let $X^x : \Omega\times[0,T]\to\mathbb{R}^d$ be the solution, in the Itô sense, of the following stochastic differential equation, 
\begin{equation}\label{eq:Itô-diffusion-sde}
    dX^x_t = \mu(X^x_t)dt + \sigma(X^x_t)dB_t, \quad X^x_0=x, \quad x\in D, t\in[0,T],
\end{equation}
where $B_t$ is a standard $d$-dimensional Brownian motion on $(\Omega, \mathcal{F}, P, (\mathbb{F}_t)_{t\in[0,T]})$. The existence of $X^x$ is guaranteed by Lemma \ref{lem:SDE-sol}. 
Dynkin's formula relates the generator  $\mathcal{F}$ of $X^x_t$, given in e.g. \cite{oksendal2003stochastic},
\begin{equation}
    \label{eq:defF}
    \F{X^x_t} = \sum_{i=1}^d \mu_i(X^x_t) (\partial_i\varphi) (X^x_t) + \frac{1}{2} \sum_{i,j,k=1}^d \sigma_{ik}(X^x_t)\sigma_{kj}(X^x_t) (\partial^2_{ij}\varphi)(X^x_t),
\end{equation}
with the initial condition $\varphi \in C^2(D)$ and differential operator $\mathcal{L}$ \eqref{eq:L} of the corresponding Kolmogorov PDE \eqref{eq:kolmogorov-pde}. 
Equipped with this notation, we state the Dynkin's formula below,
\begin{lemma}[Dynkin's formula]\label{lem:dynkin}
For every $x\in D$, let $X^x$ be the solution to a linear Kolmogorov SDE \eqref{eq:Itô-diffusion-sde} with affine $\mu:\mathbb{R}^d\to\mathbb{R}^d$ and  $\sigma:\mathbb{R}^d\to\mathbb{R}^{d\times d}$. If $\varphi\in C^{2}(\mathbb{R}^d)$ with bounded first partial derivatives, then it holds that $(\partial_t u)(x,t) = \L{u}(x,t)$ where $u$ is defined as
\begin{equation}\label{eq:dynkin}
    u(x,t) = \varphi(x) + \E{\int_0^t \F{X^x_\tau} d\tau}, \qquad \text{for }x\in D, t\in[0,T].
\end{equation}
\end{lemma}
\begin{proof}
See Corollary 6.5 and Section 6.10 in \cite{klebaner2012introduction}. %https://www.worldscientific.com/worldscibooks/10.1142/p821
%C^2_0: \cite[Thm. 7.4.1]{oksendal2003stochastic}. 
\end{proof}

Our construction of a neural network with small residual \eqref{eq:res} relies on emulating the right hand side of Dynkin's formula \eqref{eq:dynkin} with neural networks. In particular, the initial data $\varphi$ and the generator $\mathcal{F}\varphi$ will be approximated by suitable tanh neural networks. On the other hand, the expectation in \eqref{eq:dynkin} will be replaced by an accurate Monte Carlo sampling. Our construction is summarized in the following theorem,
\begin{theorem}\label{thm:approx-pinn}
Let $\alpha, \beta, \varpi, \zeta, T>0$ and let $p>2$. For every $d\in\mathbb{N}$, let $D_d = [0,1]^d$, $\varphi_d \in C^5(\mathbb{R}^d)$ with bounded first partial derivatives, let $(D_d\times [0,T],\mathcal{F},\mu)$ be a probability space and let $u_d\in C^{2,1}(D_d\times[0,T])$ be a function that satisfies
\begin{equation}
     (\partial_t u_d)(x,t) = \L{u_d}(x,t), \quad u_d(x,0) = \varphi_d(x) \quad \text{for all } (x,t) \in D_d\times [0,T]. 
\end{equation}
Moreover, assume that for every $\xi,\delta,c>0$, there exist tanh neural networks $\widehat{\varphi}_{\xi,d}: \mathbb{R}^d\to\mathbb{R}$ and $(\widehat{\mathcal{F}\varphi})_{\delta,d}: \mathbb{R}^d\to\mathbb{R}$ with respectively $\bigO(d^\alpha\xi^{-\beta})$ and $\bigO(d^\alpha \delta^{-\beta})$ neurons and weights that grow as $\bigO(d^\varpi \xi^{-\zeta})$ and $\bigO(d^\varpi \delta^{-\zeta})$ such that 
\begin{equation}\label{eq:fvarphi-acc}
    \norm{\varphi_d-\widehat{\varphi}_{\xi,d}}_{C^2(D_d)} \leq \xi \quad \text{and} \quad \norm{\mathcal{F}\varphi-(\widehat{\mathcal{F}\varphi})_{\delta,d}}_{C^2([-c,c]^d)} \leq \delta. 
\end{equation}
Then there exist constants $C,\lambda>0$ such that for every $\varepsilon>0$ and $d\in\mathbb{N}$, there exist a constant $\rho_d>0$ and a tanh neural network $\Psi_{\varepsilon,d}$ with at most $C (d\rho_d)^\lambda \varepsilon^{-\max\{5p+3, 2+p+\beta\}}$ neurons and weights that grow at most as $C (d\rho_d)^\lambda \varepsilon^{-\max\{\zeta,8p+6\}}$ for $\varepsilon\to 0$ such that
\begin{align}\label{eq:pinn-int-small}
\begin{split}
    %\left(\int_{D\times[0,T]} \abs{\frac{\partial}{\partial t} \Psi_{\varepsilon,d}(x,t) - \L{\Psi_{\varepsilon,d}}(x,t)}^2 d\mu(x,t)\right)^{\frac{1}{2}} \leq \varepsilon.
    \norm{\partial_t \Psi_{\varepsilon,d}- \L{\Psi_{\varepsilon,d}}}_{L^2(D_d\times [0,T])} + 
    \norm{\Psi_{\varepsilon,d}- u_d}_{H^1(D_d\times [0,T])} + \norm{ \Psi_{\varepsilon,d}- u_d}_{L^2(\partial(D_d\times [0,T]))}\leq \varepsilon.
\end{split}\end{align}
Moreover, $\rho_d$ is defined as 
\begin{equation}\label{eq:rhod}
    \rho_d :=  \max_{x\in D_d}\sup_{\substack{s,t \in [0,T],\\s<t}} \frac{\norm{X^x_s-X^x_t}_{\mathcal{L}^q(P, \norm{\cdot}_{\mathbb{R}^d})}}{\abs{s-t}^{\frac{1}{p}}} <\infty,  
\end{equation}
where $X^x$ is the solution, in the Itô sense, of the SDE \eqref{eq:Itô-diffusion-sde} and $q>2$ is independent of $d$. 
\end{theorem}

\begin{proof}
Based on the Dynkin's formula of Lemma \ref{lem:dynkin}, we will construct a tanh neural network, denoted by $\unn$ for some $M,N\in\mathbb{N}$, and we will prove that the PINN residual \eqref{eq:res} of $\unn$ is small. To do so, we need to define intermediate approximations $\ubar$ and $\utilde$. In this proof, $C>0$ will denote a constant that will be updated throughout and can only depend on $d$, $D$, $\mu$, $T$, $\varphi$ and $\mathcal{L}$, i.e., not on $M$ nor $N$. In particular, the dependence of $C$ on the input dimension $d$ will be of interest. We will argue that the final value of $C$ will depend polynomially on $d$ and $\rho_d$ \eqref{eq:rhod}. Because of the third point of Lemma \ref{lem:SDE-sol}, the quantity within the maximum in the definition of $\rho_d$ \eqref{eq:rhod} is finite for every individual $x\in D$ and hence the maximum of this quantity over $x\in\{0,e_1, \ldots, e_d\}$ will be finite as well. As a result of the fourth point of Lemma \ref{lem:SDE-sol} it then follows that $\rho_d<\infty$. Moreover, if $\rho_d$ depends polynomially on $d$, then so will $C$. For notational simplicity, we will not explicitly keep track of the dependence of $C$ on $d$ and $\rho_d$. 

Next, we observe that
\begin{equation}
    \max_{x\in D}\sup_{t\in [0,T]}\norm{X^x_t}_{\mathcal{L}^q(P, \norm{\cdot}_{\mathbb{R}^d})} \leq  \max_{x\in D}\sup_{t\in [0,T]}\left(\norm{x}_{\mathbb{R}^d}+ t^{\frac{1}{p}} \frac{\norm{X^x_t-x}_{\mathcal{L}^q(P, \norm{\cdot}_{\mathbb{R}^d})}}{t^{\frac{1}{p}}}\right) \leq  \max_{x\in D}\norm{x}_{\mathbb{R}^d} + (1+T^{\frac{1}{p}})\rho_d, 
\end{equation}
such that the left-hand side also grows at most polynomially in $d$ and $\rho_d$. 

Finally, we will denote by $\norm{\cdot}_2$ the norm $\norm{\cdot}_{L^2(D\times[0,T])}$ and to simplify notation we will write $u:=u_d$ and $D:=D_d$. 

\textbf{Step 1: from $u$ to $\ubar$.} In the first step we approximate the temporal integral in \eqref{eq:dynkin} by a Riemann sum, that can be readily approximated by neural networks. To this end, let $h:\mathbb{R}\to \mathbb{R}$ be defined by 
$h(x) = \max\{0,\min\{x,1\}\}$. Then we define for $N\in\mathbb{N}$,
\begin{equation}
    \ubar(x,t) = \varphi(x) + \frac{T}{N}\sum_{n=1}^N\E{h\left(\frac{Nt}{T}-n\right)\cdot \F{X^x_{\frac{nT}{N}}}}.
\end{equation}
We first define $n_0(t) = \lfloor Nt/T \rfloor$ and calculate for $t\in\left(\nO,\frac{(n_0(t)+1)T}{N}\right)$,
\begin{equation}
    \partial_t(\ubar-u) = \E{\F{X^x_\nO}-\F{X^x_t}}.
\end{equation}
Next, we make the observation that there exist constants $a_i, b_i, c_{ij}$ (that only depend on the coefficients of $\mu$ and $\sigma$) and functions $\Lambda_i$, $\Psi_i$ and $\Phi_{ij}$ (that linearly depend on $\varphi$ and its derivatives) such that
\begin{equation}
    \F{Z^x} = \sum_{i=1}^d a_i \Lambda_i(Z^x) + \sum_{i=1}^d b_i Z^x_i \Psi_i(Z^x) + \sum_{i,j=1}^d c_{ij}Z^x_i Z^x_j \Phi_{ij}(Z^x) 
\end{equation}
for any $d$-dimensional stochastic process $Z^x$. If we define $x$ to be random variable that is uniformly distributed on $D$, we can use the Lipschitz continuity of $\Lambda_i$ and the temporal regularity of $X^x$ (property (3) of Lemma \ref{lem:SDE-sol} with $\lambda\leftarrow x$) to see that
\begin{equation}
   \sup_{t\in[0,T]} \int_D \E{\abs{\Lambda_i(X^x_\nO)-\Lambda_i(X^x_t)}^2}dx \leq C\sup_{t\in[0,T]} \int_D \E{\norm{X^x_\nO-X^x_t}^2}dx \leq \frac{C}{N^{\frac{2}{p}}}.
\end{equation}
Similarly, we find using Lemma \ref{lem:SDE-sol} and the generalized Hölder inequality with $q>0$ such that $\frac{1}{p}+\frac{1}{q}=\frac{1}{2}$,
\begin{align}
\begin{split}
    &\sup_{t\in[0,T]}\left(\int_D \E{\abs{(X^x_\nO)_i\Psi_i(X^x_\nO)-(X^x_t)_i\Psi_i(X^x_t)}^2}dx\right)^{1/2}\\
    &\quad \leq \sup_{t\in[0,T]}\left(\int_D \E{\abs{(X^x_\nO)_i-(X^x_t)_i}^p}dx\right)^{1/p}\left(\int_D \E{\abs{\Psi_i(X^x_\nO)}^q}dx\right)^{1/q} \\
    & \qquad + \sup_{t\in[0,T]}\left(\int_D \E{\abs{(X^x_t)_i}^q}dx\right)^{1/q}\left(\int_D \E{\abs{\Psi_i(X^x_\nO)-\Psi_i(X^x_t)}^p}dx\right)^{1/p}\\
    &\quad \leq  \sup_{t\in[0,T]} C \left(\int_D \E{\norm{X^x_\nO-X^x_t}^p}dx\right)^{1/p} \leq \frac{C}{N^{1/p}}.
\end{split}
\end{align}
Using also the fact that
\begin{equation}
    \sup_{t\in[0,T]}\left(\int_D \E{\abs{Z^x_iZ^x_j}^q}dx\right)^{1/q} \leq  \sup_{t\in[0,T]}\left(\int_D \E{\abs{Z^x_i}^{2q}}dx\right)^{1/2q} \sup_{t\in[0,T]}\left(\int_D \E{\abs{Z^x_j}^{2q}}dx\right)^{1/2q},
\end{equation}
we can find that
\begin{equation}
    \sup_{t\in[0,T]}\left(\int_D \E{\abs{(X^x_\nO)_i(X^x_\nO)_j\Phi_{ij}(X^x_\nO)-(X^x_t)_i(X^x_t)_j\Phi_{ij}(X^x_t)}^2}dx\right)^{1/2} \leq \frac{C}{N^{1/p}}.
\end{equation}
As a result, we find that
\begin{equation}\label{eq:u-ubar-t}
    \norm{\partial_t(\ubar-u)}_2 \leq \frac{C}{N^{1/p}}.
\end{equation}
In a similar fashion, one can also find that
\begin{equation}\label{eq:u-ubar-L}
    \norm{\L{u-\ubar}}_2 \leq \frac{C}{N^{1/p}}.
\end{equation}
To obtain this result, one can use that for all $x\in\mathbb{R}^d$ and $t\in[0,T]$ it holds that
\begin{equation}
    X^x_t = \sum_{i=1}^d (X^{e_i}_t-X^0_t)x_i + X^0_t, 
\end{equation}
see Lemma \ref{lem:SDE-sol}. Using this, and writing $X^{\cdot}_t:D\to\mathbb{R}:x\mapsto X^x_t$, one can calculate that $\L{\F{X^{\cdot}_t}}(x)$ is a linear combination of terms of the form $(X^{y_1}_t)_{k_1} \cdots (X^{y_r}_t)_{k_r} F(X^x_t) G(x)$ for $y_1,\ldots,y_r\in\{0,e_1, \ldots e_d\}$, $1\leq k_1, \ldots, k_r\leq d$ (with $r$ independent of $d$) and where $F$ is a linear combination of $\varphi$ and its partial derivatives and $G$ is a product of $\mu$ and $\sigma$ and their derivatives. 
% From Lemma \ref{lem:SDE-sol} it follows for $q\geq 2$ that
% \begin{equation}
%     \max_{y\in\{0,e_1, \ldots e_d\}}\sup_{t\in [0,T]}\norm{X^y_t}_{\mathcal{L}^q(P, \norm{\cdot}_{\mathbb{R}^d})}<\infty, \quad \max_{y\in\{0,e_1, \ldots e_d\}}\sup_{s<t \in [0,T]} \frac{\norm{X^y_s-X^y_t}_{\mathcal{L}^q(P, \norm{\cdot}_{\mathbb{R}^d})}}{\abs{s-t}^{1/2}} < \infty.
% \end{equation}
Using these observations and the fact that $\rho_d<\infty$, one can obtain \eqref{eq:u-ubar-L}.
Moreover, very similar yet tedious computations yield,
\begin{equation}\label{eq:u-ubar-H1}
    \norm{u-\ubar}_{H^1(D\times [0,T])}\leq \frac{C}{N^{1/p}}.
\end{equation}

\textbf{Step 2: from $\ubar$ to $\utilde$.}
We continue the proof by constructing a Monte Carlo approximation of $\ubar$. For this purpose, we randomly draw $\omega_m \in \Omega$ for all $m\in\mathbb{N}$ and define for every $M,N\in\mathbb{N}$ the random variable
\begin{equation}
    U^{M,N}(x,t) = \varphi(x) + \frac{T}{MN}\sum_{n=1}^N\sum_{m=1}^M h\left(\frac{Nt}{T}-n\right)\cdot \F{X^x_{\frac{nT}{N}}(\omega_m)}.
\end{equation}
Using the same arguments as in the proofs of \eqref{eq:u-ubar-t} and $\eqref{eq:u-ubar-L}$, we find for all $(x,t)\in D\times[0,T]$ and $q\in\{t, x_1, \ldots x_d\}$ that,
\begin{equation}
    \E{\left(\partial_q U^{1,N}(x,t)-\E{\partial_q U^{1,N}(x,t)}\right)^2} \leq C \quad \text{and}\quad \partial_q \ubar(x,t) = \E{\partial_q U^{1,N}(x,t)}. 
\end{equation}
Invoking Lemma \ref{lem:MC}, we find that
\begin{equation}
    \E{\norm{\partial_q(U^{M,N}-u)}_2} \leq \frac{C}{\sqrt{M}}.
\end{equation}
Similarly, one can prove that
\begin{equation}
    \E{\left(\L{U^{1,N}}(x,t)-\E{\L{U^{1,N}}(x,t)}\right)^2} \leq C \quad \text{and}\quad  \L{u}(x,t) = \E{\L{U^{1,N}}(x,t)}.
\end{equation}
This can be proven using the same arguments as in the proof of \eqref{eq:u-ubar-L}. Using again Lemma \ref{lem:MC} and Lemma \ref{lem:SDE-sol}, and in combination with our previous result, we find that there is a constant $C_0>0$ independent of $M$ (and with the same properties of $C$ in terms of dependence on $d$) such that
\begin{equation}
    %\E{\max_{0\leq n\leq N}\max_{y\in\{0,e_1, \ldots e_d\}}\norm{X^y_{\frac{nT}{N}}}_{\mathbb{R}^d}+\sqrt{M}\norm{\partial_t(U^M-u)}_2+\sqrt{M}\norm{\L{U^M-u}}_2} \leq C_0
    \E{\max_{0\leq n\leq N}\max_{y\in\{0,e_1, \ldots e_d\}}\norm{X^y_{\frac{nT}{N}}}_{\mathbb{R}^d}+\sqrt{M}\norm{U^M-u}_{H^1(D\times [0,T])}+\sqrt{M}\norm{\L{U^M-u}}_2} \leq C_0
\end{equation}
and therefore by Lemma \ref{lem:exp-to-prob} that
\begin{equation}
    %\mathbb{P}\left(\max_{0\leq n\leq N}\max_{y\in\{0,e_1, \ldots e_d\}}\norm{X^y_{\frac{nT}{N}}}_{\mathbb{R}^d}+\sqrt{M}\norm{\partial_t(U^M-u)}_2+\sqrt{M}\norm{\L{U^M-u}}_2 \leq C_0 \right)>0.
    \mathbb{P}\left(\max_{0\leq n\leq N}\max_{y\in\{0,e_1, \ldots e_d\}}\norm{X^y_{\frac{nT}{N}}}_{\mathbb{R}^d}+\sqrt{M}\norm{U^M-u}_{H^1(D\times [0,T])}+\sqrt{M}\norm{\L{U^M-u}}_2 \leq C_0 \right)>0.
\end{equation}
The fact that this event has a non-zero probability implies the existence of some \textit{fixed} $\omega_m \in \Omega$, $1\leq m\leq M$, such that for the function
\begin{equation}\label{eq:utilde}
    \utilde(x,t) = \varphi(x) + \frac{T}{MN}\sum_{n=1}^N\sum_{m=1}^Mh\left(\frac{Nt}{T}-n\right)\cdot \F{X^{x}_{\frac{nT}{N}}(\omega_m)}
\end{equation}
it holds for all $1\leq m\leq M$ that
\begin{align}\label{eq:ubar-utilde}
\begin{split}
    %\norm{\partial_t(u^M-u)}_2+\norm{\L{u^M-u}}_2 &\leq \frac{C_0}{\sqrt{M}} \quad \text{and} \quad \max_{0\leq n\leq N}\max_{y\in\{0,e_1, \ldots e_d\}}\norm{X^y_{\frac{nT}{N}}(\omega_m)}_{\mathbb{R}^d} \leq C_0.
   \norm{\utilde-u}_{H^1(D\times [0,T])}+\norm{\L{\utilde-u}}_2 &\leq \frac{C_0}{\sqrt{M}} \quad \text{and} \quad \max_{0\leq n\leq N}\max_{y\in\{0,e_1, \ldots e_d\}}\norm{X^y_{\frac{nT}{N}}(\omega_m)}_{\mathbb{R}^d} \leq C_0.
    \end{split}
\end{align}

\textbf{Step 3: from $\utilde$ to $\unn$.}
For every $\epsilon>0$ and $N=N(\epsilon)\in\mathbb{N}$, let $h_\epsilon$ be a tanh neural network such that
\begin{equation}\label{eq:h-acc}
    \norm{h_\epsilon-h}_{L^\infty(\mathbb{R})}\leq \epsilon, \quad \norm{h_\epsilon'-\chi_{[0,1]}}_{L^2([-N,N])}\leq \epsilon\quad \text{and} \quad \norm{h_\epsilon'}_{L^\infty(\mathbb{R})}\leq 2, 
\end{equation}
where $\chi_{[0,1]}$ denotes the indicator function on $[0,1]$. The existence of this neural network is guaranteed by Lemma \ref{lem:h-hat}. Moreover, for $C_1 = \max_{x\in[-C_0,C_0]^d}\Fhat{x}$, we denote the multiplication operator $\times: [-2,2]\times [-2C_1,2C_1]\to\mathbb{R}: (x,y)\mapsto xy$
and every $\eta>0$, we define $\mult:[-2,2]\times [-2C_1,2C_1]\to\mathbb{R}$ to be a tanh neural network such that
\begin{equation}\label{eq:mult-acc}
    \norm{\times-\mult}_{C^2([-2,2]\times [-2C_1,2C_1])} \leq \eta. 
\end{equation}
If we now in \eqref{eq:utilde} replace $\varphi$ and $\mathcal{F}\varphi$ by $\vpa$ and $(\widehat{\mathcal{F}\varphi})_\delta$ as from \eqref{eq:fvarphi-acc}, $h$ by $h_\epsilon$ and $\times$ by $\mult$, then we end up with the tanh neural network
\begin{equation}\label{eq:unn}
    \unn(x,t) = \vpa(x) + \frac{T}{MN}\sum_{n=1}^N\sum_{m=1}^M \mult\left(h_\epsilon \left(\frac{Nt}{T}-n\right), \Fhat{X^{x,m}_{\frac{nT}{N}}}\right).
\end{equation}
A sketch of this network can be found in Figure \ref{fig:flowchart}.  
\begin{figure}

\tikzstyle{block} = [rectangle, draw, %fill=blue!20, 
    text centered, rounded corners, minimum height=3em, minimum width=2em]
\tikzstyle{circ} = [circle, draw, %fill=blue!20, 
    text centered, rounded corners, minimum height=3em, minimum width=2em]
\tikzstyle{line} = [draw, -latex']
\centering
\resizebox{0.7\textwidth}{!}{
\begin{tikzpicture}

\draw[rounded corners, gray] (0, -0.25) rectangle (6, 4.5);
\draw[rounded corners, gray] (-2.75, 1) rectangle (-0.25, 3);
\node[gray] at (-2,3.25)  {\textit{for every} $n$};
\node[gray] at (1,4.75)  {\textit{for every} $m,n$};
\node[gray] at (3.5,-1)  {\textit{affine}};

\def\up{6}
\node[block] at (5,\up) (f1) {$\unn(x,t)$};

\def\up{3.5}
\node[block] at (3,\up) (e1) {$\mult\left(h_\epsilon \left(\frac{Nt}{T}-n\right), \Fhat{X^{x,m}_{\frac{nT}{N}}}\right)$};

\def\up{2}

\node[block] at (-1.5,\up) (d1) {$h_\epsilon \left(\frac{Nt}{T}-n\right)$};
\node[block] at (3,\up) (d2) {$\Fhat{X^{x,m}_{\frac{nT}{N}}}$};
\node[block] at (8,\up) (d3) {$\vpa(x)$};

\def\up{0.5}
\node[block] at (3,\up) (c2) {$X^{x,m}_{\frac{nT}{N}}$};

\def\up{-2}
\node[circ] at (-1.5,\up) (b1) {$t$};
\node[circ] at (5,\up) (b2) {$x$};

\path [line] (b1) -- (d1);
\path [line] (b2) -- (c2);
\path [line] (c2) -- (d2);
\path [line] (d2) -- (e1);
\path [line] (d1) -- (e1);

\path [line] (b2) -- (d3);
\path [line] (d3) -- (f1);

\path [line] (e1) -- (f1);

\end{tikzpicture}
}
    % \centering
    % \includegraphics[width=\textwidth]{flow1.pdf}
\caption{Flowchart to visualize the construction of the neural network $\unn(x,t) = \vpa(x) + \frac{T}{MN}\sum_{n=1}^N\sum_{m=1}^M \mult\left(h_\epsilon \left(\frac{Nt}{T}-n\right), \Fhat{X^{x,m}_{\frac{nT}{N}}}\right)$. }
\label{fig:flowchart}
\end{figure}
In what follows, we will write $\partial_1$ for the partial derivative to the first component and we will write
\begin{equation}
    y_1 = h_\epsilon\left(\frac{Nt}{T}-n_0(t)\right), \quad y_2 = \Fhat{X^x_{\frac{n_0(t)T}{N}}(\omega_m)} \quad y_3 = \frac{Nt}{T}-n_0(t), \quad \text{and} \quad y_4 = X^x_{\frac{n_0(t)T}{N}}(\omega_m). 
\end{equation}
It holds that
\begin{align}
\begin{split}
\norm{\partial_t (\unn-\utilde)}_2 &\leq \frac{1}{M}\sum_{m=1}^M\norm{\sum_{n\neq n_0(t)}\partial_1 \mult\left(y_1,y_2\right)h_\epsilon'\left(\frac{Nt}{T}-n\right)}_2\\
& + \frac{1}{M}\sum_{m=1}^M \norm{\partial_1 \mult\left(y_1,y_2\right)h_\epsilon'\left(y_3\right)-\F{y_4}}_2.
\end{split}
\end{align}
Using \eqref{eq:mult-acc}, we find that 
\begin{align}
\begin{split}
     \frac{1}{M}\sum_{m=1}^M\norm{\sum_{n\neq n_0(t)}\partial_1 \mult\left(y_1,y_2\right)h_\epsilon'\left(\frac{Nt}{T}-n\right)}_2 &\leq CN \norm{\mult}_{C^2}\epsilon \leq CN\epsilon.
\end{split}
\end{align}
For the other term, we calculate using \eqref{eq:fvarphi-acc}, \eqref{eq:h-acc} and \eqref{eq:mult-acc} that 
\begin{align}
    \begin{split}
        &\norm{\partial_1 \mult\left(y_1,y_2\right)h_\epsilon'\left(y_3\right)-\F{y_4}}_2 \\& \leq \norm{h_\epsilon'\left(y_3\right)(\partial_1 \mult\left(y_1,y_2\right)-y_2) + h_\epsilon'\left(y_3\right)(\Fhat{y_4}-\F{y_4}) + \F{y_4}(h_\epsilon'\left(y_3\right)-\chi_{[0,1]}(y_3))}_2 \\
        &\leq C\norm{h_\epsilon'}_\infty \norm{\times - \mult}_{C^2} + C\norm{h_\epsilon'}_\infty\norm{(\widehat{\mathcal{F}\varphi})_\delta - \mathcal{F}\varphi}_{C^2} + \norm{\mathcal{F}\varphi}_\infty \norm{h'_\epsilon-\chi_{[0,1]}}_{L^2}\\
        &\leq C(\eta+\delta+\epsilon).
    \end{split}
\end{align}
Thus, we find that
\begin{equation}\label{eq:utilde-unn-t}
    \norm{\partial_t (\unn-\utilde)}_2 \leq C\left(N\epsilon + \eta+\delta\right)
\end{equation}
Finally, we obtain a bound on $ \norm{\L{\utilde-\unn}}_2$. We simplify notation again by setting
\begin{equation}
    z_1 = h_\epsilon\left(\frac{Nt}{T}-n\right), \quad z_2 = \Fhat{X^x_{\frac{nT}{N}}(\omega_m)} \quad z_3 = \frac{Nt}{T}-n, \quad \text{and} \quad z_4 = X^x_{\frac{nT}{N}}(\omega_m). 
\end{equation}
We start off by calculating
\begin{align}
    \begin{split}
\L{\utilde-\unn}  = & \L{\varphi-\vpa} + \frac{T}{MN}\sum_{m=1}^M\sum_{n=1}^N h\left(z_3\right) \cdot \L{\F{X^\cdot_{\frac{nT}{N}}(\omega_m)}}(x) \\
& - \frac{T}{MN}\sum_{m=1}^M\sum_{n=1}^N \L{\mult\left(z_1,\Fhat{X^\cdot_{\frac{nT}{N}}(\omega_m)} \right)}(x).
    \end{split}
\end{align}
Explicitly working out the above formula is straightforward, but tedious, and we omit the calculations for the sake of brevity. From this, together with a repeated use of the triangle inequality and \eqref{eq:ubar-utilde}, we find that 
\begin{align}\label{eq:utilde-unn-L}
    \begin{split}
        \norm{\L{\utilde-\unn}}_2 &\leq C \left(\norm{\varphi-\vpa}_{C^2}+\norm{\times-\mult}_{C^2}+\norm{\mathcal{F}\varphi-(\widehat{\mathcal{F}\varphi})_\delta}_{C^2}+\norm{h_\epsilon-h}_{L^\infty(\mathbb{R})}\right) \\&\leq C(\xi+\eta+\delta+\epsilon).
    \end{split}
\end{align}
Moreover, using similar tools as above we also find that
\begin{equation}\label{eq:utilde-unn-H1}
    \norm{\utilde-\unn}_{H^1(D\times [0,T])}\leq C\left(N\epsilon + \xi+\eta+\delta\right). 
\end{equation}
\textbf{Step 4: Total error bound. } From the triangle inequality and inequalities \eqref{eq:u-ubar-t}, \eqref{eq:ubar-utilde}, \eqref{eq:utilde-unn-t}, \eqref{eq:utilde-unn-L} and \eqref{eq:u-ubar-L}, we get that
\begin{align}
    \begin{split}
        \norm{\partial_t\unn-L{\unn}}_2\leq&\: \norm{\partial_t(\unn-\utilde)}_2 + \norm{\partial_t(\utilde-\ubar)}_2 + \norm{\partial_t(\ubar-u)}_2\\&+\norm{\L{u-\ubar}}_2+\norm{\L{\ubar-\utilde}}_2+\norm{\L{\utilde-\unn}}_2\\
        \leq& \: C\left(\frac{1}{N^{1/p}}+\frac{1}{\sqrt{M}}+(N\epsilon + \eta+\delta)+(\xi+\eta+\delta+\epsilon)+\frac{1}{\sqrt{M}}+\frac{1}{N^{1/p}}\right)\\
         \leq& \: C\left(\frac{1}{N^{1/p}}+\frac{1}{\sqrt{M}}+N\epsilon + \eta+\delta+\xi\right).
    \end{split}
\end{align}
Similarly, the triangle inequality together with inequalities \eqref{eq:u-ubar-H1}, \eqref{eq:ubar-utilde} and \eqref{eq:utilde-unn-H1} gives us,
\begin{equation}
\label{eq:err1}
    \norm{\unn-u}_{H^1(D\times [0,T])}\leq C\left(\frac{1}{N^{1/p}}+\frac{1}{\sqrt{M}}+N\epsilon + \eta+\delta+\xi\right).
\end{equation}
Combining this result with a multiplicative trace inequality (e.g. \cite[Theorem 3.10.1]{naep}) provides us with the result
\begin{equation}
    \norm{\unn-u}_{L^2(\partial(D\times [0,T]))}\leq C\left(\frac{1}{N^{1/p}}+\frac{1}{\sqrt{M}}+N\epsilon + \eta+\delta+\xi\right).
\end{equation}
\textbf{Step 5: network size.} Recall that we need a tanh neural network with $\bigO(d^\alpha \delta^{-\beta})$ neurons to approximate $\mathcal{F}\varphi$ to an accuracy of $\delta>0$. Similary for approximating $\varphi$, we need a tanh neural network with $\bigO(d^\alpha \xi^{-\beta})$ neurons. 

We first determine the complexity of the network sizes in terms of $\varepsilon$. The network will consist of multiple sub-networks, as illustrated in Figure \ref{fig:flowchart}. The first part constructs $M\cdot N$ copies of $(\widehat{\mathcal{F}\varphi})_\delta$, leading to a subnetwork with $\bigO\left(MN\delta^{-\beta}\right) = \bigO\left(\varepsilon^{-2-p-\beta}\right)$ neurons. Next, we need $N$ copies of $h_\epsilon$. From Lemma \ref{lem:h-hat} it follows that for each copy, one needs a subnetwork with two hidden layers of width $\bigO\left(N^{\frac{1}{2(1-\gamma)}}\epsilon^{\frac{-3}{1-\gamma}}\right)$ for any $\gamma>0$. One can calculate that $N$ copies of this lead to a width of $\bigO\left(N^{1+\frac{1}{2(1-\gamma)}}\epsilon^{\frac{-3}{1-\gamma}}\right) = \bigO\left(\varepsilon^{-5p-3}\right)$. 
The subnetwork approximating $\varphi$ consists of $\bigO(\xi^{-\beta}) = \bigO(\varepsilon^{-\beta})$ neurons. We assume that the subnetworks to approximate the identity function have a size that is negligible compared to the network sizes of the other parts \cite{deryck2021approximation}. Combining these observations with the fact that $C$ depends polynomially on $d$ and $\rho_d$, we find that there exists a constant $\lambda>0$ such that the number of neurons of the network is bounded by $\bigO((d\rho_d)^\lambda \varepsilon^{-\max\{5p+3, 2+p+\beta\}})$. 

By assumption, the weights of $(\widehat{\mathcal{F}\varphi})_\delta$ and $\widehat
\varphi_\xi$ scale as $\bigO(\varepsilon^{-\zeta})$. From \cite[Corollary 3.7]{deryck2021approximation}, it follows that the weights of $\mult$ scale as $\bigO(\varepsilon^{-1/2})$. Finally, from Lemma \ref{lem:h-hat}, the weights of $\hat{h}_\epsilon$ scale as $\bigO\left(N^{\frac{1}{(1-\gamma)}}\epsilon^{\frac{-6}{1-\gamma}}\right) = \bigO\left(\varepsilon^{-8p-6}\right)$. Hence, the weights of the total network $\unn$ grow as $\bigO\left((d\rho_d)^\lambda\varepsilon^{-\max\{\zeta,8p+6\}}\right)$, where we possibly adapted the size of $\lambda$. 
\end{proof}

\begin{remark}
For the Black-Scholes equation \eqref{eq:BS}, the initial condition is to be interpreted as a payoff function. Note that any mollified version of the payoff functions mentioned in Section \ref{sec:Kolmogorov} satisfies the regularity requirements of Theorem \ref{thm:approx-pinn}. Moreover, because of their compositional structure, these payoff functions and their derivatives can be approximated without the curse of dimensionality. Hence, the assumption \eqref{eq:fvarphi-acc} is satisfied as well. 
\end{remark}

Theorem \ref{thm:approx-pinn} reveals that the size of the constructed tanh neural network,  approximating the underlying solution $u$ of the linear Kolmogorov equation \eqref{eq:kolmogorov-pde}, and whose PINN residual is as small as desired \eqref{eq:pinn-int-small}, grows with increasing accuracy, but at a rate that is \emph{independent of the underlying dimension} $d$. Thus, it appears that this neural network overcomes the curse of dimensionality in this sense. 

However, Theorem \ref{thm:approx-pinn} reveals that the overall network size grows polynomially in $\rho_d$. It could be that this constant grows exponentially with dimension. Consequently, the overall network size will be subject to the curse of dimensionality. 
Given this issue, we will prove that at least for a subclass of Kolmogorov PDEs \eqref{eq:kolmogorov-pde}, $\rho_d$ only grows polynomially on $d$. This is for example the case when the coefficients $\mu$ and $\sigma$ are both constant functions. 

\begin{theorem}\label{thm:approx-heat}
Assume the setting of Theorem \ref{thm:approx-pinn} and assume that $\mu$ and $\sigma$ are both constant.
Then there exists a constant $\lambda>0$ such that for every $\varepsilon>0$ and $d\in\mathbb{N}$, there exists a tanh neural network $\Psi_{\varepsilon,d}$ with $\bigO(d^\lambda \varepsilon^{-\max\{5p+3, 2+p+\beta\}})$ neurons and weights that grow as $\bigO(d^\lambda \varepsilon^{-\max\{\zeta,8p+6\}})$ for small $\varepsilon$ and large $d$ such that
\begin{align}\label{eq:heat-int-small}
\begin{split}
    %\left(\int_{D\times[0,T]} \abs{\frac{\partial}{\partial t} \Psi_{\varepsilon,d}(x,t) - \L{\Psi_{\varepsilon,d}}(x,t)}^2 d\mu(x,t)\right)^{\frac{1}{2}} \leq \varepsilon.
    \norm{\partial_t \Psi_{\varepsilon,d}- \L{\Psi_{\varepsilon,d}}}_{L^2(D_d\times [0,T])} + 
    \norm{\Psi_{\varepsilon,d}- u_d}_{H^1(D_d\times [0,T])} + \norm{ \Psi_{\varepsilon,d}- u_d}_{L^2(\partial(D_d\times [0,T]))}\leq \varepsilon.
\end{split}\end{align}
\end{theorem}
\begin{proof}
We show that when $\mu$ and $\sigma$ are both constant functions, the constant $\rho_d$, as defined in \eqref{eq:rhod}, grows only polynomially in $d$. It is well-known that in this setting the solution process to the SDE \eqref{eq:Itô-diffusion-sde} is given by $X^x_t = x+ \mu t + \sigma B_t$, where $(B_t)_{t\in [0,T]}$ is a $d$-dimensional Brownian motion. The fact that $\rho_d$ only grows polynomially in $d$ then follows directly from the Lévy's modulus of continuity (Lemma \ref{lem:lévy}). The corollary then is a direct consequence of Theorem \ref{thm:approx-pinn}.
\end{proof}

\begin{remark}
We did not specify the boundary conditions explicitly in either Thoorem \ref{thm:approx-pinn} or Theorem \ref{thm:approx-heat}. The reason lies in the fact that Dynkin's formula (Lemma \ref{lem:dynkin}) holds with $\mathbb{R}^d$ as domain. Therefore, we implicitly use the trace of the true solution $u_d$ at the boundary of $D_d$ as the Dirichlet boundary condition. A similar approach has been used in e.g. \cite{grohs2018proof, hornung2020space}, where the Feynman-Kac formula is used to construct a neural network approximation for the solution to Kolmogorov PDEs. This assumption is quite reasonable as Black-Scholes type PDEs \eqref{eq:BS} are specified in the whole space. In practice, one needs to put in some artificial boundary conditions, for instance by truncating the domain. To this end, one can use some explicit formulas such as the Feynman-Kac or Dynkin's formula to (approximately) specify the boundary condition, see \cite{RT} for examples. Another possibility is to consider periodic boundary conditions as Dynkin's formula also holds in this case. 
\end{remark}

Thus, we have been able to answer question Q1 by showing that there exists a neural network, for which the PINN residual (generalization error) \eqref{eq:def-generalization-error} is as small as desired. In this process, we have also answered Q2 for this particular tanh neural network as the bound \eqref{eq:err1} clearly shows that the overall error (in the $L^2$-norm and even $H^1$-norm) of the tanh neural network $\Psi_{\varepsilon,d}$  
is arbitrarily small. 

Although in this particular case, an affirmative answer to question Q2 was a by-product of the proof of question Q1, it turns out that one can follow the recent paper \cite{MM1} and leverage the stability of Kolmogorov PDEs to answer question Q2 in much more generality, by showing that as long as the generalization error is the small, the overall error is proportionately small. We have the following precise statement about this fact,
\begin{theorem}\label{thm:L2-error}
Let $u$ be a (classical) solution to a linear Kolmogorov equation \eqref{eq:kolmogorov-pde} with $\mu\in C^1(D;\mathbb{R}^d)$ and $\sigma\in C^2(D;\mathbb{R}^{d\times d})$, $u_\theta$ a PINN and let the residuals be defined by \eqref{eq:res}. Then 
\begin{align}
\label{eq:reserr}
    \begin{split}
%\int_{D\times [0,T]} \abs{u(x,t)-u_\theta(x,t)}^2dxdt &\leq C_1 \left[\Eg^i + \Eg^t(\theta) + C_2\sqrt{\Eg^s(\theta)} + C_3 \Eg^s\right],  
\norm{u-u_\theta}_{L^2(D\times [0,T])}^2 \leq &\: C_1 \bigg[\norm{\mathcal{R}_i[u_\theta]}_{L^2(D\times [0,T])}^2 + \norm{\mathcal{R}_t[u_\theta]}_{L^2(D)}^2 \\&+ C_2\norm{\mathcal{R}_s[u_\theta]}_{L^2(\partial D \times [0,T])} + C_3\norm{\mathcal{R}_s[u_\theta]}_{L^2(\partial D \times [0,T])}^2\bigg],  
    \end{split}
\end{align}
where $C_0 = \sum_{i,j=1}^d\norm{ \partial_{ij}(\sigma\sigma^T)_{ij}}_{L^\infty(D\times[0,T])}$, $C_1 = Te^{(C_0+\norm{\mathrm{div}\mu}_\infty+1)T}$, $C_2=  \sum_{i=1}^d\norm{(\sigma\sigma^TJ_x[u-u_\theta]^T)_i}_{L^2(\partial D\times[0,T])}$ and $C_3=\norm{\mu}_\infty+\sum_{i,j,k=1}^d\norm{\partial_i (\sigma_{ik}\sigma_{jk})}_{L^\infty(\partial D\times[0,T])}$. 
\end{theorem}
\begin{proof}
Let $\hu = u_\theta-u$. Integrating $ \mathcal{R}_i[\hu](t,x)$ over $D$ and rearranging terms gives
\begin{align}\label{eq:proof-thm31-1}
    \begin{split}
        \frac{1}{2}\frac{d}{dt}\int_D \abs{\hu}^2 &= \frac{1}{2}\int_D \text{Trace}(\sigma\sigma^T H_x[\hu])\hu + \int_D \mu J_x[\hu]\hu + \int_D \mathcal{R}_i[\hu]\hu\\
%        &= \int_D \text{div}( J_x[A \hu])\hu + \frac{1}{2}\int_D \mu J_x[\hu^2] + \int_D \mathcal{R}_i[\hu]\hu\\
%        &= - \int_D J_x[A \hu])^T J_x[\hu] + \int_{\partial D} \mathcal{R}_s[u_\theta] J_x[A \hu])^T \cdot \hat{n} + \frac{1}{2}\int_D \mu J_x[\hu^2] + \int_D \mathcal{R}_i[\hu]\hu,
    \end{split}
\end{align}
where all integrals are to be interpreted as integrals with respect to the Lebesgue measure on $D$, resp. $\partial D$.
For the first term of \eqref{eq:proof-thm31-1}, we observe that $\text{Trace}(\sigma\sigma^T H_x[\hu])=\sum_{i,j,k=1}^d \sigma_{ik}\sigma_{jk}\partial_{ij}\hu$ and also that
\begin{equation}
    \int_D \partial_i (\sigma_{ik}\sigma_{jk})\hu\partial_j\hu = \int_{\partial D} \partial_i (\sigma_{ik}\sigma_{jk})\hu^2(\hat{e}_j\cdot \hat{n}) - \int_D \partial_i (\sigma_{ik}\sigma_{jk})\hu\partial_j\hu
    -\int_D \partial_{ij} (\sigma_{ik}\sigma_{jk})\hu^2
\end{equation}
for any $1\leq i,j,k\leq d$. 
Next, we define
\begin{equation}
    c_1 = 2 \sum_{i=1}^d\norm{(\sigma\sigma^TJ_x[\hu]^T)_i}_{L^2(\partial D\times[0,T])}, \: c_2 = \sum_{i,j,k=1}^d\norm{\partial_i (\sigma_{ik}\sigma_{jk})}_{L^\infty(\partial D\times[0,T])}, \: c_3 = \sum_{i,j=1}^d\norm{ \partial_{ij}(\sigma\sigma^T)_{ij}}_{L^\infty(D\times[0,T])}.
\end{equation}
From this, using integration by parts and letting $\hat{n}$ denote the unit normal on $\partial D$, we find that
\begin{align}
    \begin{split}
        &\int_D \text{Trace}(\sigma\sigma^T H_x[\hu])\hu \\&= \sum_{i,j,k=1}^d \left[\int_{\partial D} \sigma_{ik}\sigma_{jk}\hu\partial_{j}\hu (\hat{e}_i\cdot \hat{n})  - \int_D \sigma_{ik}\sigma_{jk}\partial_i\hu\partial_j\hu - \int_D \partial_i (\sigma_{ik}\sigma_{jk})\hu\partial_j\hu\right]\\
        &= \sum_{i,j,k=1}^d \left[\int_{\partial D} \sigma_{ik}\sigma_{jk}\hu\partial_{j}\hu (\hat{e}_i\cdot \hat{n})  - \int_D \sigma_{ik}\sigma_{jk}\partial_i\hu\partial_j\hu - \frac{1}{2}\int_{\partial D} \partial_i (\sigma_{ik}\sigma_{jk})\hu^2(\hat{e}_j\cdot \hat{n}) +\frac{1}{2} \int_D \partial_{ij} (\sigma_{ik}\sigma_{jk})\hu^2\right]\\
        &\leq \sum_{i=1}^d \int_{\partial D}\abs{(\sigma\sigma^T J_x(\hu)^T)_i \hu (\hat{e}_i\cdot \hat{n})} - \underbrace{\int_D J_x[\hu] \sigma (J_x[\hu]\sigma)^T}_{\geq 0} + \frac{c_2}{2} \int_{\partial D}  \abs{\mathcal{R}_s[u_\theta]}^2 + \frac{c_3}{2}\int_D \hu^2.
    \end{split}
\end{align}
For the second term of \eqref{eq:proof-thm31-1}, we find that
\begin{align}
    \begin{split}
        \int_D \mu J_x[\hu]\hu &= \frac{1}{2}\int_D \mu J_x[\hu^2] = - \frac{1}{2}\int_D \hu^2 \text{div} \mu + \frac{1}{2}\int_{\partial D} \hu^2 \mu^T \cdot \hat{n} \\&\leq \frac{1}{2} \norm{\text{div}\mu}_\infty \int_{D}  \hu^2 + \frac{1}{2} \norm{\mu}_\infty \int_{\partial D}  \abs{\mathcal{R}_s[u_\theta]}^2
        %\\&\leq  \frac{1}{2} \norm{\text{div}\mu}_\infty \norm{\hu}_\infty\int_{D} \abs{\hu} + \frac{1}{2} \norm{\mu}_\infty \int_{\partial D}  \abs{\mathcal{R}_s[u_\theta]}^2.
    \end{split}
\end{align}
Finally, we find for the third term of the right-hand side of \eqref{eq:proof-thm31-1} that
\begin{align}
    \int_D \mathcal{R}_i[\hu]\hu \leq \frac{1}{2}\int_D \mathcal{R}_i[\hu]^2 + \frac{1}{2}\int_D \hu^2 %\leq \frac{1}{2}\int_D \mathcal{R}_i[\hu]^2 + \frac{1}{2} \norm{\hu}_\infty\int_{D} \abs{\hu}.
\end{align}
Integrating \eqref{eq:proof-thm31-1} over the interval $[0,\tau]\subset[0,T]$, using all the previous inequalities together with Hölder's inequality, we find that
\begin{align}
    \begin{split}
\int_D \abs{\hu(x,\tau)}^2dx & \leq \int_D \abs{\mathcal{R}_t[u_\theta]}^2 + c_1\left(\int_{\partial D \times [0,T]}  \abs{\mathcal{R}_s[u_\theta]}^2\right)^{1/2} +   \int_{D\times[0,T]}  \abs{\mathcal{R}_i[\hu]}^2\\ 
&\quad + (c_2+\norm{\mu}_\infty) \int_{\partial D \times [0,T]}   \abs{\mathcal{R}_s[u_\theta]}^2    +(c_3+\norm{\text{div}\mu}_\infty+1)\int_{[0,\tau]}\int_D \abs{\hu(x,s)}^2dxds. 
    \end{split}
\end{align}
Using Grönwall's inequality and integrating over $[0,T]$ then gives
\begin{align}
    \begin{split}
\int_{D\times [0,T]} \abs{\hu}^2 &\leq Te^{(c_3+\norm{\text{div}\mu}_\infty+1)T} \biggl[\int_D \abs{\mathcal{R}_t[u_\theta]}^2 + c_1 \left(\int_{\partial D \times [0,T]}  \abs{\mathcal{R}_s[u_\theta]}^2\right)^{1/2} \\ &\quad + \int_{D\times[0,T]}  \abs{\mathcal{R}_i[\hu]}^2 + (c_2+\norm{\mu}_\infty) \int_{\partial D \times [0,T]}   \abs{\mathcal{R}_s[u_\theta]}^2\biggr]. 
    \end{split}
\end{align}
Renaming the constants yields the statement of the theorem. 
\end{proof}
Thus the bound \eqref{eq:reserr} clearly shows that controlling the generalization error \eqref{eq:def-generalization-error} suffices to control the $L^2$-error for the PINN, approximating the Kolmogorov equations \eqref{eq:kolmogorov-pde}. In particular, combining Theorem \ref{thm:L2-error} with Theorem \ref{thm:approx-pinn} then proves that it is possible to approximate solutions to linear Kolmogorov equations in $L^2$-norm at a rate that is independent of the spatial dimension $d$. 

It is easy to observe that the constants $C_{1,3} \sim {\mathcal O}(1)$ as they only depend on the coefficients of the Kolmogorov PDE. On the other hand, the constant $C_2$ depends on the PINN approximation $u_{\theta}$ and needs to be evaluated for each individual approximation. For instance, for the PINN $\Psi_{\epsilon,d}$, constructed in Theorem \ref{thm:approx-pinn}, it is straightforward to observe from the arguments presented in the proof of Theorem \ref{thm:approx-pinn} that $C_2 \sim {\mathcal O}(\epsilon)$.

\section{Generalization error of PINNs}\label{sec:gen}
Having answered the questions Q1 and Q2 on the smallness of the PINN residual (generalization error \eqref{eq:def-generalization-error}) and the total error for PINNs approximating the Kolmogorov PDEs \eqref{eq:kolmogorov-pde}, we turn our attention to question Q3 i.e., given small training error \eqref{eq:lf2} and for sufficiently many training samples $\S_{i,s,t}$, can one show that the generalization error \eqref{eq:def-generalization-error} (and consequently the total error by Theorem \ref{thm:L2-error}) is proportionately small?

To this end, we start with the observation that the PINN residual as well training error \eqref{eq:lf2} has three parts, two \emph{data terms} corresponding to the mismatches with the initial and boundary data and a \emph{residual term} that measures the amplitude of the PDE residual. Thus, we can embed these two types of terms in the following very general set-up: let $D\subset \mathbb{R}^d$ be compact and let $f:D\to\mathbb{R}$, $f_\theta:D\to\mathbb{R}$ be functions for all $\theta\in\Theta$. We can think of $f$ as the ground truth for the initial or boundary data for the PDE \eqref{eq:kolmogorov-pde} and $f_{\theta}$ be the corresponding restriction of approximating PINNs to the spatial or temporal boundaries. Similarly, we can think of $f\equiv 0$ as the PDE residual, corresponding to the exact solution of \eqref{eq:kolmogorov-pde} and $f_\theta$ is the \emph{interior} PINN residual (first term in \eqref{eq:res}), for a neural network with weights $\theta$. Let $M\in\mathbb{N}$ be the training set size and let $\S =\{z_1, \ldots, z_M\}\subset D^M$ be the training set, where each $z_i$ is independently drawn according to some probability measure $\mu$ on $D$. We define the (squared) training error, generalization error and empirical risk minimizer as
\begin{equation}\label{eqn:training-generalization-minimizer}
    \Et(\theta,\S)^2 = \frac{1}{M}\sum_{i=1}^M\abs{f_\theta(z_i)-f(z_i)}^2, \quad \Eg(\theta)^2 = \int_D \abs{f_\theta(z)-f(z)}^2 d\mu(z), \quad  \theta^*(\S) \in \arg \min_{\theta\in\Theta} \Et(\theta,\S)^2, 
\end{equation}
where we restrict ourselves to the (squared) $L^2$-norm only for definiteness, while claiming that all the subsequent results readily extend to general $L^p$-norms for $1 \leq p < \infty$. It is easy to see that the above set-up encompasses all the terms in the definitions of the generalization error \eqref{eq:def-generalization-error} and training error \eqref{eq:lf2} for PINNs. 

Our first aim is to decompose this very general form of generalization error in \eqref{eqn:training-generalization-minimizer} as, 
\begin{lemma}\label{lem:error-decomposition}
Let $k\in\mathbb{N}$ and $\Theta \subset \mathbb{R}^k$ compact. Then it holds that
\begin{align}\label{eq:error-decomp-2}
    \begin{split}
        \Eg(\theta^*(\S))^2 \leq&\: \sup\limits_{\substack{\theta,\vartheta\in\Theta:\\\norm{\theta-\vartheta}\leq \delta }}\abs{\Eg(\vartheta)^2-\Eg(\theta)^2} + \sup_{\theta\in\Theta}\abs{\Eg(\theta)^2-\Et(\theta)^2} \\&\:+ \sup\limits_{\substack{\theta,\vartheta\in\Theta:\\\norm{\theta-\vartheta}\leq \delta }}\abs{\Et(\theta,\S)^2-\Et(\vartheta,\S)^2}+\Et(\theta^*(\S),\S)^2. 
    \end{split}
\end{align}
\end{lemma}
\begin{proof}
Since $\Theta$ is compact, there exist for every $\delta >0$ a natural number $N=N(\delta)\in\mathbb{N}$ and parameters $\theta_1, \ldots \theta_N\in\Theta$ such that for all $\theta\in\Theta$ there exists $1\leq i\leq N$ such that $\norm{\theta-\theta_i}_\infty\leq \delta$. For every $1\leq i \leq N$ it holds that
\begin{align}
    \begin{split}
        \Eg(\theta^*(\S))^2 &\leq \abs{\Eg(\theta^*(\S))^2-\Eg(\theta_{i})^2} + \abs{\Eg(\theta_i)^2-\Et(\theta_i)^2} + \abs{\Et(\theta_i)^2-\Et(\theta^*)^2} + \Et(\theta^*(\S))^2. 
        % &= \Eg(\theta_2) - (\Et(\theta_1,\S)-\Eg(\theta_1)) + (\Et(\theta_2,\S)-\Eg(\theta_2)) +\Et(\theta_1,\S) -\Et(\theta_2,\S) \\
        % &\leq \Eg(\theta_2) + 2\max_{\theta\in\{\theta_2,\theta_1\}}\abs{\Et(\theta,\S)-\Eg(\theta)}+\Et(\theta_1,\S) -\Et(\theta_2,\S) \\
        % &\leq \Eg(\theta_2) + 2 \sup_{\theta\in\Theta}\abs{\Et(\theta,\S)-\Eg(\theta)} + \Et(\theta_1,\S),
    \end{split}
\end{align}
This error decomposition holds in particular for $i^*=i^*(\theta^*) \in \arg\min_i \norm{\theta^*-\theta_i}_\infty$. Using that $\norm{\theta^*-\theta_{i^*}}_\infty\leq \delta$ and then majorizing gives the bound from the statement. 
\end{proof}
Note that we have leveraged the compactness of the parameter space $\Theta$ in \eqref{eq:error-decomp-2} to decompose the generalization error in terms of the training error $\Et(\theta^*(\S),\S)$, the so-called \emph{generalization gap} i.e., $\sup_{\theta\in\Theta}\abs{\Eg(\theta)^2-\Et(\theta)^2} $ and error terms that measure the modulus of continuity of the generalization and training errors. From this decomposition, we can intuitively see that these error terms can be made suitably small by requiring that the generalization and training errors are, for instance, Lipschitz continuous. Then, we can use standard concentration inequalities to obtain the following \emph{very general} bound on the generalization error in terms of the training error, 

\begin{theorem}\label{thm:bound-generalization}
Let $a,c,\mathfrak{L}>0$, $k,d,M\in\mathbb{N}$, $D\subset \mathbb{R}^d$ compact, $(\Omega, \mathcal{A}, \mathbb{P})$ a probability space, $\Theta = [-a,a]^k$ and let $f:D\to \mathbb{R}$ and $f_\theta: D\to \mathbb{R}$ be functions for all $\theta\in\Theta$. Let $X_i:\Omega\to D$, $1\leq i \leq M$ be iid random variables, $\S=\{X_1, \ldots X_M\}$ and let $\theta^*(\S)$ be a minimizer of $\theta\mapsto\Et(\theta,\S)^2$. Let $\Et(\theta)^2,\Eg(\theta,\S)^2 \in [0,c]$ for all $\theta\in\Theta$ and $\S\subset D^M$ and let $\theta\mapsto\Eg(\theta)^2$ and $\theta\mapsto\Et(\theta,\S)^2$ be Lipschitz continuous with Lipschitz constant $\mathfrak{L}$. For every $\epsilon,\eta>0$, it holds that
\begin{equation}
\label{eq:gerrp}
       \Prob{\Eg(\theta^*(\S)) \leq \epsilon + \Et(\theta^*(\S),\S)} \geq 1 - \eta \quad \text{if} \quad M \geq \frac{c^2}{2\epsilon^4}\left(k\ln(\frac{2a\mathfrak{L}}{\epsilon^2})+\ln(\frac{1}{\eta})\right).
\end{equation}
\end{theorem}
\begin{proof}
For arbitrary $\epsilon>0$, set $\delta = \frac{\epsilon^2}{2\mathfrak{L}}$ and let $\{\theta_i\}_{i=1}^N$ be a $\delta$-covering of $\Theta$ with respect to the supremum norm. Then it holds that $N$ can be bounded by $(2a\mathfrak{L}/\epsilon^2)^k$ and moreover
\begin{equation}\label{eq:proof-gen-1}
    \sup_{\theta,\vartheta\in\Theta:\norm{\theta-\vartheta}\leq \delta }\abs{\Eg(\vartheta)^2-\Eg(\theta)^2} + \sup_{\theta,\vartheta\in\Theta:\norm{\theta-\vartheta}\leq \delta }\abs{\Et(\theta,\S)^2-\Et(\vartheta,\S)^2} \leq \epsilon. 
\end{equation}
Then it holds for every $1\leq i \leq N$ that 
\begin{align}\label{eq:proof-gen-2}
    \begin{split}
         \Eg(\theta^*(\S))^2 &\leq \abs{\Eg(\theta^*(\S))^2-\Eg(\theta_{i})^2} + \abs{\Eg(\theta_i)^2-\Et(\theta_i,\S)^2} + \abs{\Et(\theta_i,\S)^2-\Et(\theta^*(\S),\S)^2} + \Et(\theta^*(\S),\S)^2. 
    \end{split}
\end{align}
Next, we define a projection $\mathcal{P}:\Theta\to \Theta$ that maps $\theta$ to a unique $\theta_{i^*}$ with $i^*\in\arg\min_i \norm{\theta-\theta_i}_\infty$ and we define the following events for $1\leq i \leq N$, 
\begin{align}
    \begin{split}
        \mathcal{A} &= \left\{ \Eg(\theta^*(\S))^2 \leq \epsilon^2 + \Et(\theta^*(\S),\S)^2 \right\}, \quad
        \mathcal{B}_i = \left\{ \Eg(\theta_i)^2 \leq \epsilon^2 + \Et(\theta_i,\S)^2 \right\}, \quad
        \mathcal{C}_i = \left\{\mathcal{P}(\theta^*(\S)) = \theta_i \right\}, \\
        \mathcal{D} &= \left\{ \exists i \in\{1,\ldots, N\}: \left(\Eg(\theta_i)^2 \leq \epsilon^2 + \Et(\theta_i,\S)^2\right)\text{ and } (\mathcal{P}(\theta^*(\S)) = \theta_i)\right\}. \\
    \end{split}
\end{align}
Note that \eqref{eq:proof-gen-1} and \eqref{eq:proof-gen-2} imply that $\mathcal{D} \subseteq \mathcal{A}$ and thus $\mathbb{P}(\mathcal{D}) \leq \mathbb{P}(\mathcal{A})$. Next, by the definition of $\mathcal{P}$ it holds that $\mathcal{P}$ induces a partition on $\Theta$ and thus $\sum_i \mathcal{P}(\mathcal{C}_i) = 1$. As $\Et(\theta, \{X_i\})^2:\Omega\to [0,c]$ and $\E{\Et(\theta, \{X_i\})^2} = \Eg(\theta)^2$ for all $i$, Hoeffding's inequality (Lemma \ref{lem:hoeffding}) proves that $\Prob{\mathcal{B}_i} \geq 1- \exp(-2\epsilon^4 M/c^2)$. 
Combining this with the observation that $\mathcal{D} = \bigsqcup_{i=1}^N (\mathcal{B}_i \cap \mathcal{C}_i)$ then proves that
\begin{align}
    \begin{split}
        \mathbb{P}(\mathcal{A}) &\geq \mathbb{P}(\mathcal{D})
        = \sum_{i=1}^N \mathbb{P}(\mathcal{B}_i \cap \mathcal{C}_i)
        \geq \sum_{i=1}^N (\mathbb{P}(\mathcal{B}_i)+\mathbb{P}(\mathcal{C}_i)-\mathbb{P}(\mathcal{B}_i \cup \mathcal{C}_i))\\
        & \geq  1 + \sum_{i=1}^N (\mathbb{P}(\mathcal{B}_i)-1)
        \geq 1 - N \exp\left(\frac{-2\epsilon^4M}{c^2}\right)
        \geq 1 - \left(\frac{2a\mathfrak{L}}{\epsilon^2}\right)^k \exp\left(\frac{-2\epsilon^4M}{c^2}\right).
    \end{split}
\end{align}
As a consequence, it holds that
\begin{align}
\begin{split}
  M \geq \frac{c^2}{2\epsilon^4}\left(k\ln(\frac{2a\mathfrak{L}}{\epsilon^2})+\ln(\frac{1}{\eta})\right) \quad &\implies \quad \Prob{\Eg(\theta^*(\S))^2 \leq \epsilon^2 + \Et(\theta^*(\S),\S)^2} \geq 1 - \eta\\ 
  &\implies \quad \Prob{\Eg(\theta^*(\S)) \leq \epsilon + \Et(\theta^*(\S),\S)} \geq 1 - \eta.  
\end{split}
\end{align}

\end{proof}
The bound on the generalization error in terms of the training error \eqref{eq:gerrp} is a probabilistic statement. It can readily be recast in terms of \emph{averages} by defining the so-called \emph{cumulative} generalization and training errors of the form,
\begin{equation}\label{eqn:acc-training-error}
    \Aeg^2 = \int_{D^M}\Eg(\theta^*(\S))^2 d\mu^M(\S), \quad \Aet^2 = \int_{D^M}\Et(\theta^*(\S),\S)^2 d\mu^M(\S). 
\end{equation}
Here $\mu^M = \mu \otimes \mu \ldots \otimes \mu$ is the induced product measure on the training set $\S$. We have the following \emph{ensemble} version of Theorem \ref{thm:bound-generalization};
\begin{corollary}\label{cor:bound-generalization}
Assume the setting of Theorem \ref{thm:bound-generalization}. It holds that
\begin{equation}
    \Aeg \leq \epsilon + \Aet  \quad \text{if} \quad M \geq \frac{2c^2}{\epsilon^4}\left(k\ln(\frac{4a\mathfrak{L}}{\epsilon^2})+\ln(\frac{2c}{\epsilon^2})\right).
\end{equation}
\end{corollary}
\begin{proof}
Let $X = \Eg(\theta^*(\S))^2-\Et(\theta^*(\S),\S)^2$. Using (the last step of the proof of) Theorem \ref{thm:bound-generalization} with $\eta = \frac{\epsilon^2}{2c}$ then gives that
\begin{equation}
    \mathbb{E}[X] = \mathbb{E}[X \mathbbm{1}_{X\leq \frac{\epsilon^2}{2}}] + \mathbb{E}[X \mathbbm{1}_{X> \frac{\epsilon^2}{2}}] \leq \frac{\epsilon^2}{2} + c \Prob{X> \frac{\epsilon^2}{2}} \leq \epsilon^2, 
\end{equation}
provided that $M \geq \frac{2c^2}{\epsilon^4}\left(k\ln(\frac{4a\mathfrak{L}}{\epsilon^2})+\ln(\frac{2c}{\epsilon^2})\right)$.
\end{proof}

As a first example for illustrating the bounds of Theorem \ref{thm:bound-generalization} (and Corollary \ref{cor:bound-generalization}), we apply it to the estimation of the generalization errors, corresponding to the spatial and temporal boundaries, in terms of the corresponding training errors \eqref{eq:lf2}. These bounds readily follow from the following general bound.

\begin{corollary}\label{cor:gen-supervised}
Let $L, W\in\mathbb{N}$, $R\geq 1$, $L\geq 2$ and let $f_\theta:D\to\mathbb{R}$, $\theta\in \Theta,$ be tanh neural networks with at most $L-1$ hidden layers, width at most $W$ and weights and biases bounded by $R$. For every $0<\epsilon<1$, it holds that for the generalization and training error \eqref{eqn:training-generalization-minimizer} that,
\begin{equation}
       \Prob{\Eg(\theta^*(\S)) \leq \epsilon + \Et(\theta^*(\S),\S)} \geq 1 - \eta \quad \text{if} \quad M \geq \frac{16d(L+3)^2W^6R^4}{\epsilon^4}\ln(\frac{4\sqrt[5]{d+4}RW}{{\epsilon}}).\end{equation}
\end{corollary}

\begin{proof}
Using the inverse triangle inequality and the fact that $a^2-b^2=(a+b)(a-b)$ for $a,b\in\mathbb{R}$, we find for $\theta, \vartheta\in\Theta$ that 
\begin{align}
\begin{split}
    \abs{\int_D \abs{f_\theta(x)-f(x)}^2 - \abs{f_\vartheta(x)-f(x)}^2 d\mu(x)} &\leq 4R \int_D \abs{\abs{f_\theta(x)-f(x)} - \abs{f_\vartheta(x)-f(x)}}d\mu(x)\\
    & \leq 4R\int_D \abs{f_\theta(x)-f_\vartheta(x)}d\mu(x).
\end{split}
\end{align}
Combining this with Lemma \ref{lem:NN-lipschitz} and Lemma \ref{lem:bounds-tanh} proves that the Lipschitz constant of the map $\theta \mapsto f_\theta$ is at most $4(d+4)R^LW^{L-1}$. We can then use Corollary \ref{cor:bound-generalization} 
with $a\leftarrow R$, $\mathfrak{L}\leftarrow 4(d+4)R^{L}W^{L-1}$ and $c\leftarrow 4W^2R^2$ (from \eqref{eqn:training-generalization-minimizer}). Moreover, one can calculate that every $f_\theta$ has at most $(d+(L-2)W+1)W$ weights and $(L-1)W+1$ biases, such that $k\leftarrow 2dLW^2$. Next, we make the estimate
\begin{equation}
    \frac{c^2}{2\epsilon^4}\left(k\ln(\frac{2a\mathfrak{L}}{\epsilon^2})+\ln(\frac{2c}{\epsilon^2})\right) \leq \frac{8W^4R^4}{\epsilon^4} \cdot 2dLW^2 \ln(\frac{2^6(d+4)R^{L+3}W^{L-1}}{\epsilon^4}) \leq \frac{16d(L+3)^2W^6R^4}{\epsilon^4}\ln(\frac{4\sqrt[5]{d+4}RW}{{\epsilon}}).
\end{equation}
\end{proof}

Next, we will apply the above general results to PINNs for the Kolmogorov equation \eqref{eq:kolmogorov-pde}. 
The following corollary provides an estimate on the (cumulative) PINN generalization error and can be seen as the counterpart of Corollary \ref{cor:gen-supervised}. It is based on the fact that neural networks and their derivatives are Lipschitz continuous in the parameter vector, the proof of which can be found in Appendix \ref{app:lipschitz}. Consequently, the PINN generalization error is Lipschitz as well (cf. Lemma \ref{lem:lqq-calculation}). 

\begin{corollary}\label{cor:gen-pinn}
Let $L, W\in\mathbb{N}$, $R\geq 1$, $a,b\in\mathbb{R}$ with $a<b$ and let $u_\theta:[a,b]^d\to \mathbb{R}$, $\theta\in \Theta,$ be tanh neural networks with smooth activation function $\sigma$, at most $L-1$ hidden layers, width at most $W$ and weights and biases bounded by $R$. For $q=i,t,s$ let the PINN generalization $\Eg^q$ and training $\Et^q$ errors for linear Kolmogorov PDEs (cf. Section \ref{sec:Kolmogorov}) and let $c_q>0$ be such that $\Et^q(\theta)^2,\Eg^q(\theta,\S)^2 \in [0,c_q]$ for all $\theta\in\Theta$ and $\S\subset D^M$. Assume that $\max\{\norm{\varphi}_\infty, \norm{\psi}_\infty\}\leq \max_{\theta\in\Theta}\norm{u_\theta}_\infty$ and define the constants
\begin{align}
    \begin{split}
        \alpha &= \max\{1,\abs{a}, \abs{b}, \norm{\sigma}_\infty\},\qquad \beta= \max\{1, \norm{\sigma'}_\infty,\norm{\sigma''}_\infty,\norm{\sigma'''}_\infty\},\\
        C&=\max_{x\in D}\left(1+\sum_{i=1}^d\abs{\mu(x)_i}+\sum_{i,j=1}^d\abs{(\sigma(x)\sigma(x)^*)_{ij}}\right).
    \end{split}
\end{align}
Then for any $\epsilon>0$ it holds that
\begin{equation}\label{eq:bound-Mq}
 \Aeg^q \leq \epsilon +\Aet^q \quad \text{ if }  M_q \geq \frac{24dL^2W^2c_q^2}{\epsilon^4}\ln(4c_qRW\beta\sqrt[6]{\frac{C(d+7)}{\epsilon^2}}).
\end{equation}

\end{corollary}
\begin{proof}
Setting $C=\max_{x\in D}\left(1+\sum_{i=1}^d\abs{\mu(x)_i}+\sum_{i,j=1}^d\abs{(\sigma(x)\sigma(x)^*)_{ij}}\right)$, we can use Corollary \ref{cor:bound-generalization} with $a\leftarrow R$, $c\leftarrow c_q$, $\mathfrak{L}\leftarrow 2^5 C^2(d+7)^2L^4R^{6L-1}W^{6L-6}\beta^{2L}$ (cf. Lemma \ref{lem:lqq-calculation}) and $k\leftarrow 2dLW^2$ (cf. proof of Corollary \ref{cor:gen-supervised}). We then calculate
\begin{equation}
    k\ln(\frac{4a\mathfrak{L}}{\epsilon^2})+\ln(\frac{2c_q}{\epsilon^2}) \leq 6kL \ln(4c_qRW\beta\sqrt[6]{\frac{C(d+7)}{\epsilon^2}}) = 12dL^2W^2\ln(4c_qRW\beta\sqrt[6]{\frac{C(d+7)}{\epsilon^2}}).
\end{equation}
\end{proof}

\begin{remark}\label{rem:cq}
Corollary \ref{cor:gen-pinn} requires bounds $c_q$ on the training errors $\Et^q$ and the generalization errors $\Eg^q$ of the PINN. Lemma \ref{lem:lqq-calculation} provides such bounds, given by $c_i = 4 \alpha C(d+7)L^2R^{3L}W^{3L-3}\beta^L$ and $c_t=c_s =2 WR$. Although the values for $c_t$ and $c_s$ are of reasonable size, the value for $c_i$ is likely to be a large overestimate. It might makes sense to consider the approximation
\begin{equation}
    c_i \approx \max_{n,m} \Et^i(\theta_n,\{x_m\})
\end{equation}
for some randomly sampled $\theta_n\in\Theta$ and $x_m\in D$.
\end{remark}

Combining Corollary \ref{cor:gen-pinn} with Theorem \ref{thm:L2-error} allows us to bound the $L^2$-error of the PINN in terms of the (cumulative) training error and the training set size. The following corollary proves that a well-trained PINN on average has a low $L^2$-error provided that the training set is large enough. It is also possible to prove a similar probabilistic statement instead of a statement that holds on average. 

\begin{corollary}\label{cor:L2-error}
Let $u$ be a (classical) solution to a linear Kolmogorov equation \eqref{eq:kolmogorov-pde} with $\mu\in C^1(D;\mathbb{R}^d)$ and $\sigma\in C^1(D;\mathbb{R}^{d\times d})$, $u^* = u_{\theta^*(\S)}$ a trained PINN, let $\Aet^i, \Aet^s$ and $\Aet^t$ denote the interior, spatial and temporal cumulative training error, cf. \eqref{eq:def-generalization-error} and let $C_1$, $C_2$ and $C_3$ be the constants as defined in Theorem \ref{thm:L2-error}. If the training set sizes where chosen as in \eqref{eq:bound-Mq} of Corollary \ref{cor:gen-pinn} for some $\epsilon>0$, then 
\begin{align}\label{eq:cor-L2}
    \begin{split}
\int_{(D\times [0,T])^M}\int_{D\times [0,T]} \abs{u(x,t)-u_{\theta^*(\S)}(x,t)}^2dxdt d\mu^M(\S) 
&\leq C_1 \left[(\Aet^i)^2 + (\Aet^t)^2 + C_2(\Aet^s+\sqrt{\epsilon}) + C_3 (\Aet^s)^2 +(C_3+2)\epsilon\right].
    \end{split}
\end{align}
\end{corollary}
\begin{proof}
This is a direct consequence of Corollary \ref{cor:gen-pinn} and the proof of Theorem \ref{thm:L2-error} (in particular, one needs to take the expectation of all training sets $\S$ before applying Hölder's inequality in the proof of Theorem \ref{thm:L2-error}). 
\end{proof}

\begin{remark}
If we assume that the optimization algorithm used to minimize the training loss finds a global minimum, then one can prove that the cumulative training errors in \eqref{eq:cor-L2} are small if the training set is large enough.  
To see this, first observe that for the network $\Psi_{\varepsilon, d}$ that was constructed in Theorem \ref{thm:approx-pinn} it holds that $\Eg^i(\theta_\Psi), \Eg^s(\theta_\Psi)$ and $\Eg^t(\theta_\Psi)$ are all of order $\bigO(\varepsilon)$.
Since $\Psi_{\varepsilon, d}$ is not correlated with the training data $\S$, one can use a Monte Carlo argument to find for any $\epsilon>0$ that
\begin{equation}
 \Et^q(\theta_\Psi) \leq \epsilon + \Eg^q(\theta_\Psi) \quad \text{ if }  M_q \sim c_q^2\epsilon^{-2}
\end{equation}
and as a consequence that $ \Et^q(\theta_\Psi) = \bigO(\epsilon)$. 
If the optimization algorithm reaches a global minimum, the training loss of $u_{\theta^*(\S)}$ will be upper bounded by that of $\Psi_{\varepsilon, d}$. Therefore it also holds that $ \Aet^q = \bigO(\epsilon)$. 
\end{remark}

Thus, in Corollaries \ref{cor:gen-pinn} and \ref{cor:L2-error}, we have answered the question Q3 by proving that a small training error and a sufficiently large number of samples, as chosen in \eqref{eq:bound-Mq}, suffice to ensure a small generalization error (and total error). Moreover, the number of samples only depends polynomially on the dimension. Therefore, it overcomes the \emph{curse of dimensionality}.

\section{Discussion}
\label{sec:5}
Physics informed neural networks (PINNs) are widely used in approximating both forward as well as inverse problems for PDEs. However, there is a paucity of rigorous theoretical results on PINNs that can explain their excellent empirical performance. In particular, one wishes to answer the questions Q1 (on the smallness of PINN residuals), Q2 (smallness of the total error) and Q3 (smallness of the generalization error if the training error is small) in order to provides rigorous guarantees for PINNs. 

In this article, we aimed to address these theoretical questions rigorously. We do so within the context of the Kolmogorov equations, which are linear parabolic PDEs of the general form \eqref{eq:kolmogorov-pde}. The heat equation as well as the Black-Scholes equation of option pricing are prototypical examples of these PDEs. Moreover, these PDEs can be set in very high-dimensional spatial domains. Thus, in addition to providing rigorous bounds on the PINN generalization error and total error, we also aimed to investigate whether PINNs can overcome the curse of dimensionality in this context. 

To this end, we answered question Q1 in Theorem \ref{thm:approx-pinn}, where we constructed a PINN (see Figure \ref{fig:flowchart}) for which the PINN residual (generalization error) can be made as small as possible. Our constuction relied on emulating Dynkin's formula \eqref{eq:dynkin}. Under suitable assumptions on the initial data as well as on the underlying stochastic process (cf. \eqref{eq:rhod} and Theorem \ref{thm:approx-heat}), we are also able to prove that the size of the constructed only grew polynomially, in input spatial dimension. Thus, we were able to show that this PINN was able to overcome the curse of dimensionality in attaining as small a residual as desired. 

Next, we answered question Q2 in Theorem \ref{thm:L2-error} by leveraging the stability of Kolmogorov PDEs to bound the total error (in $L^2$) for PINNs in terms of the underlying generalization error. 

Finally, question Q3 that required one to bound the generalization error in terms of the training error was answered by using an error decomposition, Lipschitz continuity of the underlying generalization and training error maps and concentration inequalities in Corollary \ref{cor:gen-pinn}, where we derived a bound on the generalization error in terms of the training error and for sufficiently many randomly chosen training samples \eqref{eq:bound-Mq}. Moreover, the number of training samples only grew polynomially in the dimension, alleviating the curse of dimensionaly in this regard. 

% We combined our theoretical analysis to derive an estimate for the total $L^2$-error for PINNs in corollary \ref{cor:L2-error2}, where we show that the training process can result in a PINN, whose overall error in approximating the solution of the Kolmogorov PDE \eqref{eq:kolmogorov-pde}, can be made as small as desired. Furthermore, the size of this PINN as well as the total number of training samples do not suffer from the curse of dimensionality. Thus, we provide a comprehensive error analysis for PINNs in approximating a large and important class of PDEs. 

Although we do not present numerical experiments in this paper, we point the readers to \cite{RT} and the forthcoming paper \cite{MMT1}, where a large number of numerical experiments for PINNs in approximating both forward and inverse problems for Kolmogorov type and related equations, are presented. In particular, these experiments reveal that PINNs overcome the curse of dimensionality in this context. These findings are consistent with our theoretical results. 

At this stage, it is instructive to contrast our results with related works. As mentioned in the introduction, there are very few papers where PINNs are rigorously analyzed. When comparing to \cite{shin2020convergence}, we highlight that the fact that the authors of \cite{shin2020convergence} used a special bespoke H\"older-type regularization term that penalized the gradients in their loss function. In practice, one trains PINNs in the $L^2$ (or $L^1$) setting and it is unclear how relevant the assumptions of \cite{shin2020convergence} are in this context. On the other hand, we use the natural training paradigm for PINNs and prove rigorously that overall errors can be made small. Comparing with \cite{MM1}, we observe that the authors of \cite{MM1} only address questions Q2 and (partially) Q3, but in a very general setting. It is not proved in \cite{MM1} that the total error can be made small. We do so here. Moreover, we also provide the first bounds for PINNs, where the curse of dimensionality is alleviated. 

It is an appropriate juncture to compare our results with a large number of articles demonstrating the alleviation of the curse of dimensionality for neural networks approximating Kolmogorov type PDEs, see \cite{grohs2018proof,berner2020analysis} and references therein. We would like to point out that these articles consider the \emph{supervised learning} paradigm, where (possibly large amounts of) data needs to be provided to train the neural network for approximating solutions of PDEs. This data has to be generated by either expensive numerical simulations or the use of representation formulas such as the Feynman-Kac formulas, which requires solutions of underlying SDEs. In contrast, we recall that PINNs do not require \emph{any data} in the interior of the domain and thus are very diferent in design and conception to supervised learning frameworks. 

We would also like to highlight some limitations of our analysis. We showed in Theorem \ref{thm:approx-pinn} that network size in approximating solutions of general Kolmogorov equations \eqref{eq:kolmogorov-pde} depended on the rate of growth the quantity $\rho_d$, defined in \eqref{eq:rhod}. We were also able to prove in Theorem \ref{thm:approx-heat} that $\rho_d$ only grew polynomially (in dimension) for a subclass of Kolmogorov PDEs. Extending these results to general Kolmogorov PDEs is an open question. Moreover, it is worth repeating (see Remark \ref{rem:cq}) that the constants in our estimates are clearly not optimal and might be significant overestimates, see \cite{MM1} for a discussion on this issue. 

Finally, we point out that although we focussed our results on the large and important class of Kolmogorov PDEs in this paper, the methods that we developed will be very useful in the analysis of PINNs for approximating PDEs. In particular, the use of smoothness of the underlying PDEs solutions and their approximation by Tanh neural networks (as in \cite{deryck2021approximation}), to build PINNs with small PDE residuals can be applied to a variety of linear and non-linear PDEs. Similarly, the error decomposition \eqref{eq:error-decomp-2} and Theorem \ref{thm:bound-generalization} (Corollary \ref{cor:bound-generalization}) are very general and can be used in many different contexts, to bound PINN generalization error by training error, for sufficiently many random training points. We plan to apply these techniques for the comprehensive error analysis of PINNs for approximating forward as well as inverse problems for PDEs in forthcoming papers.

\bibliographystyle{abbrv}
\bibliography{ref}

\appendix

\section{Additional material for Section \ref{sec:3}
}\label{app:approximation}

\subsection{Auxiliary results}\label{app:approx-aux}

\begin{lemma}\label{lem:MC1}
Let $p\in[2,\infty)$, $d,m\in\mathbb{N}$, let $(\Omega, \mathcal{F}, \mathcal{P})$ be a probability space, and let $X_i:\Omega\to\mathbb{R}^d, i\in\{1,\ldots, m\}$, be i.i.d. random variables with $\E{\norm{X_1}}<\infty$. Then it holds that
\begin{equation}
    \left(\E{\norm{\E{X_1}-\frac{1}{m}\sum_{i=1}^m X_i}^p}\right)^{1/p} \leq 2 \sqrt{\frac{p-1}{m}}\left(\E{\norm{\E{X_1}-X_1}^p}\right)^{1/p}. 
\end{equation}
\end{lemma}
\begin{proof}
This result is \cite[Corollary 2.5]{grohs2018proof}.
\end{proof}
\begin{lemma}\label{lem:MC}
Let $p\in[2,\infty)$, $q,m\in\mathbb{N}$, let $(\Omega, \mathcal{F}, \mathcal{P})$ and $(\mathcal{D}, \mathcal{A}, \mu)$ be probability spaces, and let for every $q\in\mathcal{D}$ the maps $X_i^q:\Omega\to\mathbb{R}, i\in\{1,\ldots, m\}$, be i.i.d. random variables with $\E{\abs{X^q_1}}<\infty$. Then it holds that
\begin{equation}
    \E{\left(\int_{\mathcal{D}}\abs{\E{X^q_1}-\frac{1}{m}\sum_{i=1}^m X^q_i}^p\mu(dq)\right)^{1/p}} \leq 2\sqrt{\frac{p-1}{m}} \left(\int_{\mathcal{D}}\E{\abs{\E{X^q_1}-X_1^q}^p}\mu(dq)\right)^{1/p}.
\end{equation}
\end{lemma}
\begin{proof}
The proof involves Hölder's inequality, Fubini's theorem and Lemma \ref{lem:MC1}. The calculation is as in \cite[eq. (226)]{grohs2018proof}.
\end{proof}

\begin{lemma}\label{lem:exp-to-prob}
Let $\epsilon>0$, let $(\Omega, \mathcal{F}, \mathcal{P})$ be a probability space, and let $X:\Omega\to\mathbb{R}$ be a random variable that satisfies $\E{\abs{X}}\leq \epsilon$. Then it holds that $\mathbb{P}(\abs{X}\leq \epsilon)>0$. 
\end{lemma}
\begin{proof}
This result is \cite[Proposition 3.3]{grohs2018proof}.
\end{proof}

\begin{lemma}[Lévy's modulus of continuity]\label{lem:lévy}
For $(B_t)_{t\in [0,1]}$ a Brownian motion, it holds almost surely that
\begin{equation}
    \limsup_{h \downarrow 0} \sup_{0\leq t \leq 1 - h} \frac{\abs{B_{t+h}-B_t}}{\sqrt{2h\log(1/h)}}=1.
\end{equation}
\end{lemma}
\begin{proof}
This result is due to \cite{levy1954theorie} and can be found in most probability theory textbooks.
\end{proof}

\begin{lemma}\label{lem:SDE-sol}
Let $T>0$, $p\geq 2$, $d,m\in\mathbb{N}$, let $(\Omega, \mathcal{F}, P, (\mathbb{F}_t)_{t\in[0,T]})$ be a stochastic basis and let $W:[0,T]\times\Omega\to \mathbb{R}^m$ be a standard $m$-dimensional Brownian motion on $(\Omega, \mathcal{F}, P, (\mathbb{F}_t)_{t\in[0,T]})$. Let $\lambda\in \mathcal{L}^p(P|_{\mathbb{F}_0}, \norm{\cdot}_{\mathbb{R}^d})$ and let $\mu:\mathbb{R}^d\to \mathbb{R}^d$ and $\sigma:\mathbb{R}^d\to \mathbb{R}^{d\times m}$ be affine functions. Then there exists an up to indistinguishability unique $(\mathbb{F}_t)_{t\in[0,T]}$-adapted stochastic process $X^\lambda:[0,T]\times \Omega\to \mathbb{R}^d$, which satisfies
\begin{enumerate}
    \item  that for all $t\in[0,T]$ it holds $P$-a.s. that
    \begin{equation}
        X^\lambda_t = \lambda +\int_0^t\mu(X^\lambda_s)ds + \int_0^t\sigma(X^\lambda_s)dW_s
    \end{equation}
    \item it holds that $\sup_{t\in [0,T]}\norm{X^\lambda_t}_{\mathcal{L}^p(P, \norm{\cdot}_{\mathbb{R}^d})}<\infty$, 
    \item it holds that for all $\alpha\in (0,\frac{1}{2}]$ that
    \begin{equation}
        \sup_{\substack{s,t \in [0,T],\\s<t}} \frac{\norm{X^\lambda_s-X^\lambda_t}_{\mathcal{L}^p(P, \norm{\cdot}_{\mathbb{R}^d})}}{\abs{s-t}^\alpha} < \infty, 
    \end{equation}
    \item for all $x\in\mathbb{R}^d$, $t\in[0,T]$ and $\omega\in\Omega$ it holds that
    \begin{equation}
    X^x_t(\omega) = \sum_{i=1}^d \left(X^{e_i}_t(\omega)-X^0_t(\omega)\right)x_i + X^0_t(\omega).
\end{equation}
\end{enumerate}
\end{lemma}
\begin{proof}
Properties (1)-(3) are proven in \cite[Theorem 4.5.1]{nasode}. Property (4) follows from Lemma 2.20 in \cite{grohs2018proof} and Lemma 3.3 in \cite{berner2020analysis}.
\end{proof}

\begin{lemma}\label{lem:h-hat}
Let $h: \mathbb{R}\to\mathbb{R} :x\mapsto \min\{1,\max\{0,x\}\}$. For every $N\geq 2$ and $\epsilon,\gamma>0$ there exists a tanh neural network $\hat{h}$ with two hidden layers, $\bigO\left(N^{\frac{1}{2(1-\gamma)}}\epsilon^{\frac{-3}{1-\gamma}}\right)$ neurons and weights growing as $\bigO\left(N^{\frac{1}{(1-\gamma)}}\epsilon^{\frac{-6}{1-\gamma}}\right)$ such that
\begin{equation}
    \norm{h-\hat{h}}_{L^{\infty}(\mathbb{R})} \leq \epsilon , \quad \norm{h'-\hat{h}'}_{L^{2}([-N,N])} \leq \epsilon \quad \text{ and }\quad \norm{\hat{h}'}_{L^{\infty}(\mathbb{R})}\leq 2. 
\end{equation}
\end{lemma}
\begin{proof}
We first approximate $h$ with a function $\Tilde{h}$ that is twice continuously differentiable, 
\begin{equation}
    \Tilde{h}(x) = \begin{cases}
    0 &x\leq -\frac{\pi \epsilon^2}{2},\\
    \frac{1}{2}\left(\frac{\pi \epsilon^2}{2} + x - \epsilon^2 \cos(\frac{x}{\epsilon^2})\right) &-\frac{\pi \epsilon^2}{2} < x \leq \frac{\pi \epsilon^2}{2}, \\
    x & \frac{\pi \epsilon^2}{2} < x \leq 1 - \frac{\pi \epsilon^2}{2},\\
    \frac{1}{2}\left(1-\frac{\pi \epsilon^2}{2} + x + \epsilon^2 \cos(\frac{1-x}{\epsilon^2})\right) & 1 - \frac{\pi \epsilon^2}{2} < x \leq 1 + \frac{\pi \epsilon^2}{2}, \\
    1 & 1 + \frac{\pi \epsilon^2}{2} < x. 
    \end{cases}
\end{equation}
It is easy to prove that $\norm{h-\Tilde{h}}_{L^\infty(\mathbb{R})} = \bigO(\epsilon^2)$. Next, we calculate the derivative of $\Tilde{h}$,
\begin{equation}
    \Tilde{h}'(x) = \begin{cases}
    0 &x\leq -\frac{\pi \epsilon^2}{2},\\
    \frac{1}{2}\left(1+ \sin(\frac{x}{\epsilon^2})\right) &-\frac{\pi \epsilon^2}{2} < x \leq \frac{\pi \epsilon^2}{2}, \\
    1 & \frac{\pi \epsilon^2}{2} < x \leq 1 - \frac{\pi \epsilon^2}{2},\\
     \frac{1}{2}\left(1+ \sin(\frac{1-x}{\epsilon^2})\right) & 1 - \frac{\pi \epsilon^2}{2} < x \leq 1 + \frac{\pi \epsilon^2}{2}, \\
    0 & 1 + \frac{\pi \epsilon^2}{2} < x. 
    \end{cases}
\end{equation}
A straightforward calculation leads to the bound $\norm{h'-\Tilde{h}'}_{L^2(\mathbb{R})} = \bigO(\epsilon)$. Finally, one can easily check that $\Tilde{h}''$ is continuous and that $\norm{\Tilde{h}''}_{L^\infty(\mathbb{R})} = \bigO(\epsilon^{-2})$. An application of \cite[Theorem 5.1]{deryck2021approximation} on $\Tilde{h}$ gives us for every $\gamma>0$ and $N$ large enough the existence of a tanh neural network $\hat{h}^\mathcal{N}$ with two hidden layers and $\bigO(\mathcal{N})$ neurons for which it holds that $\norm{\Tilde{h}-\hat{h}^\mathcal{N}}_{W^{1,\infty}([-1,2])} = \bigO(N^{-1+\gamma}\epsilon^{-2})$. Because of the nature of the construction of $\hat{h}^\mathcal{N}$, the monotonous behaviour of the hyperbolic tangent towards infinity and the fact that $\Tilde{h}$ is constant outside $[-1,2]$, the stronger result that $\norm{\Tilde{h}-\hat{h}^\mathcal{N}}_{W^{1,\infty}(\mathbb{R})} = \bigO(\mathcal{N}^{-1+\gamma}\epsilon^{-2})$ holds automatically as well. As a result we find that $\norm{(\hat{h}^\mathcal{N}) '}_{L^{\infty}(\mathbb{R})}\leq 2$, $\norm{\Tilde{h}-\hat{h}^\mathcal{N}}_{L^{\infty}(\mathbb{R})} = \bigO(\mathcal{N}^{-1+\gamma}\epsilon^{-2})$ and $\norm{\Tilde{h}-\hat{h}^\mathcal{N}}_{L^{2}([-N,N])} = \bigO(\sqrt{N}\mathcal{N}^{-1+\gamma}\epsilon^{-2})$. If we choose $\mathcal{N}\sim N^{\frac{1}{2(1-\gamma)}}\epsilon^{\frac{-3}{1-\gamma}}$ then we find that
\begin{equation}
    \norm{\Tilde{h}-\hat{h}^\mathcal{N}}_{L^{\infty}(\mathbb{R})} \leq \epsilon \quad \text{ and }\quad \norm{\Tilde{h}'-(\hat{h}^\mathcal{N})'}_{L^{2}([-N,N])} \leq \epsilon. 
\end{equation}
Moreover, \cite[Theorem 5.1]{deryck2021approximation} tells us that the weights of $\hat{h}^\mathcal{N}$ grow as $\bigO(\mathcal{N}^2) = \bigO\left(N^{\frac{1}{(1-\gamma)}}\epsilon^{\frac{-6}{1-\gamma}}\right)$. The statement then follows from applying the triangle inequality. 
\end{proof}

\section{Lipschitz continuity in the parameter vector of a neural network and its derivatives}\label{app:lipschitz}

In this section we will prove that for any $x\in D$, a neural network and its corresponding Jacobian and Hessian matrix are Lipschitz continuous in the parameter vector. This property is of crucial importance to find bounds on the generalization error of physics informed neural networks, cf. Section \ref{sec:gen}. We first introduce some notation and then state or results. The main results of this section are Lemma \ref{lem:NN-lipschitz} and Lemma \ref{lem:total-lipschitz}. 

We denote by  $\sigma:\mathbb{R}\to\mathbb{R}$ be an (at least) twice continuously differentiable activation function, like tanh or sigmoid. For any $n\in\mathbb{N}$, we write for $x\in\mathbb{R}^n$ that $\sigma(x) := (\sigma(x_1), \ldots, \sigma(x_n))$. We use the definition of a neural network as in Definition \ref{def:nn-app}.

Recall that for a differentiable function $f:\mathbb{R}^n\to\mathbb{R}^m$ the Jacobian matrix $J[f]$ is defined by 
\begin{equation}
    J[f]_{ij} = \frac{\partial f_i}{\partial x_j} \in \mathbb{R}^{m\times n}. 
\end{equation}
For our purpose, we make the following the following convention. For any $1\leq k\leq L$, we define
\begin{equation}
    J_k^\theta(x) := J[f_k^\theta]\left((f_{k-1}^\theta\circ \cdots \circ f_1^\theta)(x)\right) \in \mathbb{R}^{l_k \times l_{k-1}}. 
\end{equation}
Similarly, for a twice differentiable function $g:\mathbb{R}^n\to\mathbb{R}$ the Hessian matrix is defined by
\begin{equation}
    H[g]_{ij} = \frac{\partial^2 g}{\partial x_i \partial x_j}. 
\end{equation}
Slightly abusing notation, we generalize this to vector-valued functions $g:\mathbb{R}^n\to\mathbb{R}^m$. We write
\begin{equation}
    H[g]_{kij} = \frac{\partial^2 g_k}{\partial x_i \partial x_j}, 
\end{equation}
where we identify $\mathbb{R}^{1\times n \times n}$ with $\mathbb{R}^{n \times n}$ to make the definitions consistent. Similarly, if $v\in\mathbb{R}^{1 \times m}$, then $v\cdot H[g]$ should be interpreted as
\begin{equation}
    v\cdot H[g](x) := \sum_{k=1}^m v_k H[g_k](x) \in \mathbb{R}^{n \times n}. 
\end{equation}
For any $1\leq k < L$, we write
\begin{equation}
    H^\theta_k(x) := H[f_{k}^\theta]\left((f_{k-1}^\theta\circ \cdots \circ f_1^\theta)(x)\right) \in \mathbb{R}^{l_k \times l_{k-1} \times l_{k-1}}. 
\end{equation}
Finally, we will use the notation $J^\theta := J[\Psi^\theta]$ and $H^\theta := H[\Psi^\theta]$. The following lemma presents a generalized version of the chain rule. 

\begin{lemma}\label{lem:chain-rule}
Let $f:\mathbb{R}^n\to\mathbb{R}^m$ and $g:\mathbb{R}^m\to\mathbb{R}$. Then it holds that
\begin{equation}
    H[g\circ f](x):= J[f](x)^T \cdot H[g](f(x)) \cdot J[f](x) + J[g](f(x)) \cdot H[f](x). 
\end{equation}
\end{lemma}
We now apply this formula to find an expression for $H^\theta$ in terms of $ J^\theta_k$ and $ H^\theta_k$. 
\begin{lemma}\label{lem:jac-hess-NN}
It holds that
\begin{equation}
    J[\Psi^\theta] = \prod_{k=0}^{L-1} J_{L-k}^\theta \quad \text{and} \quad  H[\Psi^\theta] = \sum_{k=1}^L (J^\theta_1)^T\cdots(J^\theta_{k-1})^T\cdot\left(J^\theta_{L}\cdots J^\theta_{k+1}\cdot H^\theta_{k}\right)\cdot J^\theta_{k-1} \cdots J^\theta_{1}. 
\end{equation}
\end{lemma}

\begin{proof}
The first statement is just the chain rule for calculating the derivative of a composite function. We prove the second statement using induction. For the base step, let $L=1$. Then $\Psi^\theta = f_L^\theta$ and we have $H[\Psi^\theta] = H^\theta_L$. For the induction step, take $K\in\mathbb{N}, K\geq 2$ and assume that the statement holds for $L=K-1$. Now let $\Phi^\theta = f_{K}^\theta\circ \cdots \circ f_2^\theta$ and $\Psi^\theta = \Phi^\theta \circ f_1^\theta$. Applying the generalized chain rule to calculate $H[\Phi^\theta \circ f_1^\theta]$ and using the induction hypothesis on $H[\Phi^\theta]$ gives the wanted result. 
\end{proof}

Next, we formally introduce the element-wise supremum norm $\infn{\cdot}$. Let $N\in\mathbb{N}$, $n_0, \ldots n_N\in\mathbb{N}$ and $A\in\mathbb{R}^{n_{1}\times \cdots \times n_N}$. Then we define
\begin{equation}\label{eq:supremum-norm}
    \infn{A} := \max_{1\leq i_1 \leq n_1} \cdots \max_{1\leq i_N \leq n_N} \abs{A_{i_1 \cdots i_N}}. 
\end{equation}
Let $R>0$ and suppose that $A_i\in \mathbb{R}^{n_{i-1} \times n_i}$. Then it holds that 
\begin{equation}
    \infn{\prod_{i=1}^N A_i} \leq \infn{A_N}\prod_{i=1}^{N-1} n_i \infn{A_i}. 
\end{equation}
Moreover, for $v\in\mathbb{R}^{1 \times a}$ and $A\in\mathbb{R}^{a\times b \times c}$ it holds that $\infn{v\cdot A} \leq a \infn{v}\infn{A}$.

The following lemma states that the output of each layer of a neural network is Lipschitz continuous in the parameter vector for any input $x\in[a,b]^d$. The lemma is stated for neural networks with a differentiable activation function, but can be easily adapted for e.g. ReLU neural networks. 

\begin{lemma}\label{lem:NN-lipschitz}
Let $d,L,W\in\mathbb{N}$ with $L,W\geq 2$, $a,b\in\mathbb{R}$ with $a<b$ and $R\geq 1$. Moreover, let $\theta, \vartheta \in \Theta_{L,W,R}$, $\alpha = \max\{1,\abs{a}, \abs{b}, \norm{\sigma}_\infty\}$ and $\beta= \max\{1, \norm{\sigma'}_\infty\}$. Then it holds for $1\leq K\leq L$ that
\begin{equation}
    \norm{f_{K}^\theta \circ \cdots \circ f_1^\theta-f_{K}^\vartheta \circ \cdots \circ f_1^\vartheta}_{L^\infty([a,b]^d)} \leq \alpha(d+4) W^{K-1} R^{K-1} \beta^K\infn{\theta-\vartheta}.
\end{equation}
\end{lemma}
\begin{proof}
Let $l_0, \ldots, l_L$ denote the widths of the neural network, where $l_0=d$. Let $x\in[a,b]^d$ be arbitrary. First of all, it holds that
\begin{align}
    \begin{split}
        \infn{f_1^\theta(x)-f_1^\vartheta(x)} &= \infn{\sigma(W^\theta_1x+b^\theta_1)-\sigma(W^\vartheta_1x+b^\vartheta_1)}\\
        &\leq \norm{\sigma'}_\infty \infn{(W^\theta_1-W^\vartheta_1)x+(b^\theta_1-b^\vartheta_1)}\\
        &\leq \beta (d\alpha+1)\infn{\theta-\vartheta}.
    \end{split}
\end{align}
Now let $2\leq k \leq L$ and define $y = (f_{k-1}^\theta\circ \cdots \circ f_1^\theta)(x)$ and $\Tilde{y} = (f_{k-1}^\vartheta\circ \cdots \circ f_1^\vartheta)(x)$. We find that
\begin{align}
    \begin{split}
         \infn{f_k^\theta(y)-f_k^\vartheta(\Tilde{y})} 
         & \leq \max\{1,\norm{\sigma'}_\infty\} \infn{(W^\theta_k-W^\vartheta_k)y+b^\theta_k-b^\vartheta_k+W^\vartheta_k(y-\Tilde{y})}\\
          & \leq \beta((l_{k-1}\alpha + 1)\infn{\theta-\vartheta} + l_{k-1}R\infn{y-\Tilde{y}}).
    \end{split}
\end{align}
A recursive application of this inequality then gives us for $1\leq K \leq L$ that
\begin{align}
    \begin{split}
         & \norm{f_{K}^\theta\circ f_{K-1}^\theta \circ \cdots \circ f_1^\theta-f_{K}^\vartheta\circ f_{K-1}^\vartheta \circ \cdots \circ f_1^\vartheta}_\infty \\
         & \leq \sum_{k=1}^K l_{K-1} \cdots l_k(l_{k-1}\alpha+1)R^{K-k}\beta^{K-k+1}\infn{\theta-\vartheta}\\
         & \leq  W^{K-1}(d\alpha+1) R^{K-1}\beta^K\infn{\theta-\vartheta} +\beta (W\alpha+1)\infn{\theta-\vartheta}\sum_{k=2}^K W^{K-k} R^{K-k}\beta^{K-k}\\
         & \leq   W^{K-1}(d\alpha+1) R^{K-1}\beta^K\infn{\theta-\vartheta} + \frac{\beta (W\alpha+1)W^{K-1} R^{K-1}\beta^{K-1}}{WR\beta-1}\infn{\theta-\vartheta}\\
         & \leq \alpha(d+4) W^{K-1} R^{K-1}\beta^{K}\infn{\theta-\vartheta},
    \end{split}
\end{align}
where we used that $\beta(W\alpha+1)/(WR\beta-1) \leq \beta(2\alpha+1)/(2R\beta-1) \leq 3\alpha$ when $W\geq 2, R\geq 1, \alpha\geq 1$. 

\end{proof}

\begin{lemma}\label{lem:jac-hess-lipschitz}
Let $d,L,W\in\mathbb{N}$ with $L,W\geq 2$, $a,b\in\mathbb{R}$ with $a<b$ and $R\geq 1$. Moreover, let $\theta, \vartheta \in \Theta_{L,W,R}$, $\alpha = \max\{1,\abs{a}, \abs{b}, \norm{\sigma}_\infty\}$ and $\beta= \max\{1, \norm{\sigma'}_\infty,\norm{\sigma''}_\infty,\norm{\sigma'''}_\infty\}$. Then it holds for all $1\leq k \leq L$ and $x\in[a,b]^d$ that
\begin{align}
     \infn{J^\theta_k(x)_i-J^\vartheta_k(x)_i} &\leq \beta(1 +\alpha(d+4) W^{k-1} R^{k}\beta^{k-1} + R(\alpha W+1)) \infn{\theta-\vartheta} \quad \text{and}\\
       \infn{H^\theta_k(x)_i-H^\vartheta_k(x)_i} &\leq 2\beta R(1+\alpha(d+4) W^{k-1} R^{k}\beta^{k-1} + R(\alpha W+1)) \infn{\theta-\vartheta}.
\end{align}
\end{lemma}
\begin{proof}
Let $w_i^T$ be the $i$-th row  of $W^{\theta, k}$, let $\Tilde{w}_i^T$ be the $i$-th row of $W^{\vartheta, k}$ and set $b:=b^{\theta, k}$ and $\Tilde{b}:=b^{\vartheta, k}$. Let $F = f_{k-1}^\theta\circ \cdots \circ f_1^\theta$ and $\Tilde{F} = f_{k-1}^\vartheta\circ \cdots \circ f_1^\vartheta$. For $1\leq i \leq l_k$, we have that
\begin{align}
    J^\theta_k(x)_i &= \sigma'(w_i^T\cdot F(x)+b_i) \cdot w_i^T \in \mathbb{R}^{1 \times l_{k-1}}\\
     H^\theta_k(x)_i &= \sigma''(w_i^T\cdot F(x)+b_i) \cdot w_i \cdot w_i^T \in \mathbb{R}^{l_{k-1} \times l_{k-1}}
\end{align}
and analogously for $J^\vartheta_k(x)_i$ and $H^\vartheta_k(x)_i$. The triangle inequality and the Lipschitz continuity of $\sigma'$ gives us that
\begin{align}
    \begin{split}
         \infn{J^\theta_k(x)_i-J^\vartheta_k(x)_i} \leq& \: \norm{\sigma'}_\infty \infn{w_i -\Tilde{w}_i } + \abs{\sigma'(w_i^T\cdot F(x)+b_i)-\sigma'(\Tilde{w}_i^T\cdot \Tilde{F}(x)+\Tilde{b}_i)} \infn{\Tilde{w}_i}\\
         \leq & \: \beta\infn{\theta-\vartheta} + \norm{\sigma''}_\infty R \abs{w_i^T\cdot (F(x)-\Tilde{F}(x)) +(w_i-\Tilde{w}_i)^T\cdot \Tilde{F}(x) +b_i-\Tilde{b}_i}\\
         \leq&\: \beta \infn{\theta-\vartheta} + \norm{\sigma''}_\infty R \left(l_{k-1}R\infn{F(x)-\Tilde{F}(x)}+(l_{k-1}\norm{\sigma}_\infty+1)\infn{\theta-\vartheta}\right).
    \end{split}
\end{align}
Using that $\infn{F(x)-\Tilde{F}(x)} \leq \alpha(d+4) W^{k-2} R^{k-2}\beta^{k-1}\infn{\theta-\vartheta}$ (Lemma \ref{lem:NN-lipschitz}) for $k\geq 2$ and $l_{k-1}\leq W$, we get
\begin{equation}
     \infn{J^\theta_k(x)_i-J^\vartheta_k(x)_i} \leq \beta(1 +\alpha(d+4) W^{k-1} R^{k}\beta^{k-1} + R(\alpha W+1)) \infn{\theta-\vartheta}
\end{equation}
for $k\geq 2$. One can check that the inequality also holds for $k=1$. 

For the Hessian matrix, the triangle inequality and the Lipschitz continuity of $\sigma''$ gives us that
\begin{align}
    \begin{split}
         \infn{H^\theta_k(x)_i-H^\vartheta_k(x)_i} \leq& \norm{\sigma''}_\infty \infn{w_i \cdot w_i^T-\Tilde{w}_i \cdot \Tilde{w}_i^T} + \abs{\sigma''(w_i^T\cdot F(x)+b_i)-\sigma''(\Tilde{w}_i^T\cdot \Tilde{F}(x)+\Tilde{b}_i)} \infn{\Tilde{w}_i \cdot \Tilde{w}_i^T}\\
         \leq & 2\beta R\infn{\theta-\vartheta} + \norm{\sigma'''}_\infty R^2(\alpha W+1) \infn{\theta-\vartheta} + \norm{\sigma'''}_\infty R^3 W \infn{F(x)-\Tilde{F}(x)}
    \end{split}
\end{align}
Using Lemma \ref{lem:NN-lipschitz} again, we get
\begin{equation}
     \infn{H^\theta_k(x)_i-H^\vartheta_k(x)_i} \leq 2\beta R(1+\alpha(d+4) W^{k-1} R^{k}\beta^{k-1} + R(\alpha W+1)) \infn{\theta-\vartheta}
\end{equation}
for $k\geq 2$. One can check that the inequality also holds for $k=1$. 
\end{proof}

The following lemma states that the Jacobian and Hessian matrix of a neural network are Lipschitz continuous in the parameter vector for any input $x\in[a,b]^d$. 

\begin{lemma}\label{lem:total-lipschitz}
Let $d,L,W\in\mathbb{N}$ with $L,W\geq 2$, $a,b\in\mathbb{R}$ with $a<b$ and $R\geq 1$. Moreover, let $\theta, \vartheta \in \Theta_{L,W,R}$, $\alpha = \max\{1,\abs{a}, \abs{b}, \norm{\sigma}_\infty\}$ and $\beta= \max\{1, \norm{\sigma'}_\infty,\norm{\sigma''}_\infty,\norm{\sigma'''}_\infty\}$. Then it holds that for all $x\in[a,b]^d$ that
\begin{align}
 \infn{J[\Psi^\theta](x)-J[\Psi^\vartheta](x)} &\leq 2\alpha(d+7)LR^{2L-1}W^{2L-2}\beta^{L-1}\infn{\theta-\vartheta},\\
    \infn{H[\Psi^\theta](x)-H[\Psi^\vartheta](x)} &\leq 4\alpha(d+7)L^2 R^{3L-1}W^{3L-3}\beta^L \infn{\theta-\vartheta}.
\end{align}
\end{lemma}
\begin{proof}
We will prove the formulas by repeatedly using the triangle inequality and using the representations proven in Lemma \ref{lem:jac-hess-NN}. To do so, we need to introduce some notation. Define for $0\leq l\leq L+k-1$ the object $\phi^{l}\in\{\theta, \vartheta\}^{2L}$ such that 
\begin{equation}
    \phi^{l}_j = \begin{cases}\vartheta & j\leq l,\\ \theta & j > l. \end{cases} \qquad \text{and} \qquad A^{k,l}_j = \begin{cases}
    (J^{\phi^{l}_j}_j)^T & 1\leq j \leq k-1,\\ 
    J^{\phi^{l}_{k}}_{L+k-j} & k\leq j \leq L-1 \\
    H^{\phi^{l}_{L}}_{k} & j=L \\
    J^{\phi^{l}_{k}}_{L+k-j} & L+1 \leq j \leq L+k-1.
    \end{cases}
\end{equation}
In particular, $\phi^{k,0}_j = \theta$ and $\phi^{k,{L+k-1}}_j = \vartheta$ for all $j$. To simplify notation, we write
\begin{align}
    \begin{split}
         h_k^{l} &= (J^{\phi^{l}_1}_1)^T\cdots(J^{\phi^{l}_{k-1}}_{k-1})^T\cdot\left(J^{\phi^{l}_{k}}_{L}\cdots J^{\phi^{l}_{L-1}}_{k+1}\cdot H^{\phi^{l}_{L}}_{k}\right)\cdot J^{\phi^{l}_{L+1}}_{k-1} \cdots J^{\phi^{l}_{L+k-1}}_{1} = \prod_{j=1}^{L+k-1} A^{k,l}_j.
    \end{split}
\end{align}
The triangle inequality and Lemma \ref{lem:jac-hess-NN} then give that
\begin{equation}
    \infn{H^\theta-H^\vartheta} \leq \sum_{k=1}^L \sum_{l=1}^{L+k-1} \infn{ h_k^{l-1}- h_k^{l}}. 
\end{equation}
Observe that $A^{k,l-1}_j-A^{k,l}_j = 0$ for $j\neq l$. Therefore
\begin{align}
    \begin{split}
        \infn{h_k^{l-1}- h_k^{l}} & = \infn{A^{k,l}_1 \cdots A^{k,l}_{l-1}\cdot( A^{k,l-1}_{l}-A^{k,l}_{l}) \cdot A^{k,l}_{l+1} \cdots A^{k,l}_{L+k-1}} \\
        &\leq (l_1 \cdots l_{k-1})^2 \cdot l_k \cdots l_{L-1} \cdot R^{L+k-2} \infn{A^{k,l-1}_{l}-A^{k,l}_{l}}\\
        &\leq W^{L+k-2}R^{L+k-2} \infn{A^{k,l-1}_{l}-A^{k,l}_{l}}.
    \end{split}
\end{align}
From Lemma \ref{lem:jac-hess-lipschitz}, it follows that
\begin{equation}
    \infn{A^{k,l-1}_{l}-A^{k,l}_{l}} \leq 2\beta R(1+\alpha(d+4) W^{k-1} R^{k}\beta^{k-1} + R(\alpha W+1)) \infn{\theta-\vartheta}
\end{equation}
Writing $\gamma := 1+ R(\alpha W+1)$ we get
\begin{align}
    \begin{split}
         \infn{H^\theta-H^\vartheta} &\leq \sum_{k=1}^L (L+k-1) W^{L+k-2}R^{L+k-2}\cdot 2\beta R(1+\alpha(d+4) W^{k-1} R^{k}\beta^{k-1} + R(\alpha W+1)) \infn{\theta-\vartheta} \\
         &\leq \sum_{k=1}^L 2L W^{2L-2}R^{2L-2}\cdot 2\beta R\alpha(d+7) W^{L-1} R^{L}\beta^{L-1}\infn{\theta-\vartheta}\\
         &\leq 4\alpha(d+7)L^2 R^{3L-1}W^{3L-3}\beta^L \infn{\theta-\vartheta}.
    \end{split}
\end{align}
In an entirely similar fashion we obtain 
\begin{align}
    \infn{J^\theta-J^\vartheta} \leq \sum_{k=1}^L W^{L-1}R^{L-1}\infn{J^\theta_k-J^\vartheta_k} \leq 2\alpha(d+7)LR^{2L-1}W^{2L-2}\beta^{L-1}\infn{\theta-\vartheta}.
\end{align}
\end{proof}

\section{Additional material for Section \ref{sec:gen}
}\label{app:gen}

\begin{lemma}[Hoeffding's inequality]\label{lem:hoeffding}
Let $\epsilon, c>0$, $N\in\mathbb{N}$, let $(\Omega,\mathcal{A},\mathbb{P})$ be a probability space and let $X_n:\Omega\to [0,c]$ be independent random variables. Then it holds that
\begin{equation}
    \mathbb{P}\left(\frac{1}{N}\left(\sum_{i=1}^N(X_i-\mathbb{E}[X_i])\right) \geq \epsilon \right) \leq  \exp\left(\frac{-2\epsilon^2N}{c^2}\right). 
\end{equation}
\end{lemma}

\begin{lemma}\label{lem:bounds-tanh}
Let $x\in\mathbb{R}$ and $ \sigma(x) = \tanh{x} = \frac{e^{-x}-e^{x}}{e^{-x}+e^{x}}$. It holds that $\sigma'(x) = 1-(\sigma(x))^2$ and $\sigma''(x) = -2\sigma(x)/(1-(\sigma(x))^2)$. In addition, it holds that $\norm{\sigma'}_\infty = 1$ and $\norm{\sigma''}_\infty = 4/3\sqrt{3} \leq 1$ and $\norm{\sigma'''}_\infty = 2$. 
\end{lemma}

The following lemma provides estimate on the various PINN residuals. It is based on the fact that neural networks and their derivatives are Lipschitz continuous in the parameter vector, the proof of which can be found in Appendix \ref{app:lipschitz}. 

\begin{lemma}\label{lem:lqq-calculation}
Let $d, L, W\in\mathbb{N}$, $R\geq 1$, $a,b\in\mathbb{R}$ with $a<b$ and let $u_\theta:[a,b]^d\to \mathbb{R}$, $\theta\in \Theta,$ be tanh neural networks with smooth activation function $\sigma$, at most $L-1$ hidden layers, width at most $W$ and weights and biases bounded by $R$. Let the PINN generalization $\Eg^q$ and training $\Et^q$ errors be defined as in Section \ref{sec:PINNs} for linear Kolmogorov PDEs (cf. Section \ref{sec:Kolmogorov}). Let $\alpha = \max\{1,\abs{a}, \abs{b}, \norm{\sigma}_\infty\}$ and $\beta= \max\{1, \norm{\sigma'}_\infty,\norm{\sigma''}_\infty,\norm{\sigma'''}_\infty\}$ and assume that $\max\{\norm{\varphi}_\infty, \norm{\psi}_\infty\}\leq \max_{\theta\in\Theta}\norm{u_\theta}_\infty$. Let $\mathfrak{L}^q_Q$ denote the Lipschitz constant of $\mathcal{E}^q_Q$, for $q=i,t,s$ and $Q=G,T$. Then it holds that
\begin{equation}
    \mathfrak{L}^q_Q \leq 2^5 \max_{x\in D}\left(1+\sum_{i=1}^d\abs{\mu(x)_i}+\sum_{i,j=1}^d\abs{(\sigma(x)\sigma(x)^*)_{ij}}\right)^2(d+7)^2L^4R^{6L-1}W^{6L-6}\beta^{2L}.
\end{equation}
\end{lemma}

\begin{proof}
Without loss of generality, we only focus on $\Eg^q$, for $q=i,s,t$. 
We see for $q=i,t,s$
\begin{equation}
    \abs{\Eg^q(\theta)-\Et^q(\vartheta)}_\infty \leq 2 \max_\theta \norm{\mathcal{R}_q[u_\theta]}_\infty \norm{\mathcal{R}_q[u_\theta]-\mathcal{R}_q[\Phi^\vartheta]}_\infty
\end{equation}
For $q=t,s$ and $(x,t)\in D\times [0,T]$, it follows from Lemma \ref{lem:NN-lipschitz} that
\begin{equation}
    \abs{\mathcal{R}_q[u_\theta](x,t)-\mathcal{R}_q[\Phi^\vartheta](t,x)} \leq (d+4) W^{L-1} R^{L-1} \infn{\theta-\vartheta},
\end{equation}
% \begin{equation}
%     \abs{\Et(\theta,(t,x))-\Et(\vartheta,(t,x))} \leq 2 \max_\theta \norm{\mathcal{R}_i[u_\theta]}_\infty \abs{\mathcal{R}_i[u_\theta](t,x)-\mathcal{R}_i[\Phi^\vartheta](t,x)}
% \end{equation}
and similarly using Lemma \ref{lem:total-lipschitz} that
\begin{align}
    \begin{split}
\abs{\mathcal{R}_i[u_\theta](t,x)-\mathcal{R}_i[\Phi^\vartheta](t,x)} &\leq (1+\abs{\mu(x)}_1)\infn{J^\theta-J^\vartheta} + \abs{\sigma(x)\sigma(x)^*}_1 \infn{H_x^\theta-H_x^\vartheta}\\
&\leq 4 \alpha (1+\abs{\mu(x)}_1+\abs{\sigma(x)\sigma(x)^*}_1)(d+7)L^2R^{3L-1}W^{3L-3}\beta^L\infn{\theta-\vartheta}, 
    \end{split}
\end{align}
where we let $\abs{\cdot}_p$ denote the vector $p$-norm of the vectorized version of a general tensor (cf. \eqref{eq:supremum-norm}).
Next, we calculate using again Lemma \ref{lem:total-lipschitz} (by setting $\vartheta=0$) and $\max\{\norm{\varphi}_\infty, \norm{\psi}_\infty\}\leq \max_{\theta\in\Theta}\norm{u_\theta}_\infty$ for $q=t,s$ that
\begin{align}
    \begin{split}
\max_\theta \norm{\mathcal{R}_i[u_\theta]}_\infty &\leq4 \alpha C(d+7)L^2R^{3L}W^{3L-3}\beta^L, \quad  \max_\theta \norm{\mathcal{R}_q[u_\theta]}_\infty \leq 2 WR, 
    \end{split}
\end{align}
where $C=\max_{x\in D}(1+\abs{\mu(x)}_1+\abs{\sigma(x)\sigma(x)^*}_1)$. Combining all the previous results prove the stated bound.
\end{proof}

\end{document}